\documentclass[10pt]{amsart}

\usepackage[colorlinks]{hyperref}
\usepackage{color,graphicx,shortvrb}
\usepackage[latin 1]{inputenc}

\usepackage[active]{srcltx} 

\usepackage{enumerate}
\usepackage{amssymb}
\usepackage{tikz-cd}

\usepackage [ all ]{xy}

\newtheorem{theorem}{Theorem}[section]

\newtheorem{corollary}[theorem]{Corollary}

\newtheorem{lemma}[theorem]{Lemma}

\newtheorem{proposition}[theorem]{Proposition}

\def\J#1#2#3{ \left\{ #1,#2,#3 \right\} }
\def\RR{{\mathbb{R}}}
\def\NN{{\mathbb{N}}}
\def\11{\textbf{$1$}}
\def\CC{{\mathbb{C}}}

\def\KK{{\mathbb{K}}}
\def\TT{{\mathbb{T}}}

\DeclareGraphicsExtensions{.jpg,.pdf,.png,.eps}

\begin{document}

\title[On the Mazur--Ulam property for $C(K,H)$]{On the Mazur--Ulam property for the space of Hilbert-space-valued continuous functions}

\author[M. Cueto-Avellaneda, A.M. Peralta]{Mar{\'i}a Cueto-Avellaneda, Antonio M. Peralta}

\address[M. Cueto-Avellaneda, A.M. Peralta]{Departamento de An{\'a}lisis Matem{\'a}tico, Facultad de
Ciencias, Universidad de Granada, 18071 Granada, Spain.}
\email{mcueto@ugr.es, aperalta@ugr.es}




\date{}

\begin{abstract}
Let $K$ be a compact Hausdorff space and let $H$ be a real or complex Hilbert space with dim$(H_\RR)\geq 2$. We prove that the space $C(K,H)$ of all $H$-valued continuous functions on $K$, equipped with the supremum norm, satisfies the Mazur--Ulam property, that is, if $Y$ is any real Banach space, every surjective isometry $\Delta$ from the unit sphere of $C(K,H)$ onto the unit sphere of $Y$ admits a unique extension to a surjective real linear isometry from $C(K,H)$ onto $Y$. Our strategy relies on the structure of $C(K)$-module of $C(K,H)$ and several results in JB$^*$-triple theory. For this purpose we determine the facial structure of the closed unit ball of a real JB$^*$-triple and its dual space.
\end{abstract}

\maketitle
\thispagestyle{empty}

\section{Introduction}
A Banach space $X$ satisfies the \emph{Mazur--Ulam property} if for any Banach space $Y$, every surjective isometry $\Delta: S(X)\to S(Y)$ admits an extension to a surjective real linear isometry from $X$ onto $Y$, where $S(X)$ and $S(Y)$ denote the unit spheres of $X$ and $Y$, respectively. This property, which was first named by L. Cheng and Y. Dong in \cite{ChenDong2011}, is equivalent to say that Tingley's problem (see \cite{Ting1987}) admits a positive solution for every surjective isometry from $S(X)$ onto the unit sphere of any other Banach space.\smallskip

Behind their simple statements, Tingley's problem and the Mazur--Ulam property are hard problems which remain unsolved even for surjective isometries between the unit spheres of a couple of two dimensional normed spaces (the reader is invited to take a look to the recent papers \cite{WH19} and \cite{CabSan19}, where this particular case is treated). Positive solutions to Tingley's problem have been found for surjective isometries $\Delta: S(X)\to S(Y)$ when $X$ and $Y$ are von Neumann algebras \cite{FerPe17d}, compact C$^*$-algebras \cite{PeTan16}, atomic JBW$^*$-triples \cite{FerPe17c}, spaces of trace class operators \cite{FerGarPeVill17}, spaces of $p$-Schatten von Neumann operators with $1\leq p\leq \infty$ \cite{FerJorPer2018}, preduals of von Neumann algebras and the self-adjoint parts of two von Neumann algebras \cite{Mori2017}. The surveys \cite{Ding2009,YangZhao2014}, and \cite{Pe2018} are appropriate references to the reader in order to check the state-of-the-art of this problem.\smallskip

Apart from a wide list of classical Banach spaces satisfying the Mazur--Ulam property (cf. \cite{YangZhao2014,Pe2018}), new achievements prove that this property is satisfied by commutative von Neumann algebras \cite{CuePer}, unital complex C$^*$-algebras and real von Neumann algebras \cite{MoriOza2018}, and more recently, JBW$^*$-triples with rank one or rank bigger than or equal to three \cite{BeCuFerPe2018}. The latest two mentioned references naturally lead us to consider the Mazur--Ulam property on the space $C(K,H)$ of all continuous functions from a compact Hausdorff space $K$ into a real or complex Hilbert space $H$. This space is not, in general, a C$^*$-algebra nor a JBW$^*$-triple because it is neither a dual Banach space. However, it possesses a motivating structure of Hilbert $C(K)$-module, and consequently, a structure of JB$^*$-triple, where $C(K)$ stands for the space $C(K,\mathbb{C})$.\smallskip

The main conclusions in this note prove that for every real Hilbert space $\mathcal{H}$ with dim$(\mathcal{H})\geq 2$ and every compact Hausdorff space $K$, the real Banach space $C(K,\mathcal{H})$ satisfies the Mazur--Ulam property (see Corollaries \ref{c Mazur-Ulam for real Hilber even or infinite dimensional} and \ref{c Mazur-Ulam for real Hilber odd finite dimensional}). We previously establish in Theorem \ref{t MazurUlam property for C(K,h)} that the same property holds when we consider the complex Banach space of continuous functions with values in a complex Hilbert space $H$, showing that each surjective isometry $\Delta : S(C(K,H))\to S(Y)$, where $Y$ is a real Banach space, extends to a surjective real linear isometry from $C(K,H)$ onto $Y$. R. Liu proved in \cite[Corollary 6]{Liu2007} that the space $C(K,\mathbb{R}),$ of all real continuous functions on $K,$ satisfies the Mazur--Ulam property.\smallskip

Our strategy relies on the natural JB$^*$-triple structure associated with the space $C(K,H)$. This structure provides the key tools and results to pursue our goals. We would like to vindicate the usefulness of techniques in JB$^*$-triple theory to solve natural problems in functional analysis. In subsection \ref{subsec: background} we gather a basic background, definitions and results on JB$^*$-triple theory required in this note.\smallskip

The paper is structured in five sections, this first one serves as introduction and the fifth and last section contains the main conclusions. In section \ref{sec:2} we try to illustrate the fact that the unit sphere of $C(K,H)$ is metrically distinguishable from the unit sphere of a unital C$^*$-algebra and from the unit sphere of a real von Neumann algebra. More precisely, we prove in Theorem \ref{t C(K,H) cannot be isometrically isomorphic to the sphere of a C*-algebra} that for any complex Hilbert space $H$ with dimension bigger than or equal to 2, there exists no surjective isometry from the unit sphere of $C(K,H)$ onto the unit sphere of a C$^*$-algebra. Moreover, for a real Hilbert space $\mathcal{H}$ with dim$(\mathcal{H})=3$ or dim$(\mathcal{H})\geq 5$, there exists no surjective isometry from the unit sphere of $C(K,\mathcal{H})$ onto the unit sphere of a real von Neumann algebra (cf. Theorem \ref{t C(K,H) cannot be isometrically isomorphic to the sphere of a real W*-algebra}).\smallskip

One of the most successful tools applied in recent studies on the Mazur--Ulam property is derived from an accurate knowledge of the facial structure of the closed unit ball of one of the involved Banach spaces. Weak$^*$-closed faces of the closed unit ball of a JBW$^*$-triple and norm-closed faces of the closed unit ball of its predual are well known thanks to the studies of C.M. Edwards and G.T. R\"{u}ttimann \cite{EdRutt88}. Norm-closed faces of the closed unit ball of a general JB$^*$-triple and weak$^*$-closed faces of the closed unit ball of its dual space were completely determined in \cite{EdFerHosPe2010,FerPe10}. Edwards and R\"uttimann enlarged our knowledge with the description of the weak$^*$-closed faces of the closed unit ball of a real JBW$^*$-triple, and of the norm-closed faces of the unit ball of its predual (cf. \cite{EdRutt2001}). Until now the structure of norm-closed faces of a general real JB$^*$-triple remains unexplored, we shall devote section \ref{sec: facial real} to culminate the study of the facial structure of the closed unit ball of a real JB$^*$-triple and its dual space.\smallskip

On the other hand, it is irrefutable that the extremal structure of Banach spaces has become a focus of attention, either as main topic or as a helpful tool for understanding the underlying geometry. In the setting of unital C$^*$-algebras the Russo--Dye theorem asserts that the closure of the convex hull of the unitary elements is the closed unit ball. M. Mori and N. Ozawa prove in \cite[Theorem 2]{MoriOza2018} that every Banach space $X$ such that the closed convex hull of the extreme points of its closed unit ball has non-empty interior satisfies that every convex body $\mathcal{K}\subset X$ has the strong Mankiewicz property, that is, every surjective isometry $\Delta$ from $\mathcal{K}$ onto an arbitrary convex subset $L$ in a normed space $Y$ is affine. This is a key ingredient to prove that unital C$^*$-algebras, real von Neumann algebras and JBW$^*$-triples of rank 1 or bigger or equal than three satisfy the Mazur--Ulam property \cite{MoriOza2018,BeCuFerPe2018}.\smallskip

In Section \ref{sec: 3} we revisit some results in \cite{Phelps65,Cantwell68,Peck1968} to establish a Krein--Milman type theorem showing that for any compact Hausdorff space $K,$ and every real Hilbert space $\mathcal{H}$ with dim$(\mathcal{H})\geq 2$, the closed unit ball of $C(K,\mathcal{H})$ coincides with the closed convex hull of its extreme points (cf. Proposition \ref{p conditions on C(K,H) to satisfy the Strong Mankiewicz property}). We also prove that, for each real Hilbert space $\mathcal{H}$ with dimension bigger than or equal to 2, every element in a maximal norm-closed proper face of the closed unit ball of ${C(K,\mathcal{H})}$ can be approximated in norm by a finite convex combination of elements in that face which are also extreme points of the closed unit ball of $C(K,\mathcal{H})$ (see Corollaries \ref{c convex combinations of unitaries in maximal faces} and \ref{c convex combinations of unitaries in maximal faces real}). We further prove that certain real JB$^*$-subtriples of $C(K,H)$ satisfy the strong Mankiewicz property (cf. Propositions \ref{p N strong Mankiewicz property} and \ref{p N strong Mankiewicz property rel Hilbert spaces}).

\subsection{Basic background in JB$^*$-triple theory}\label{subsec: background}

We recall that, according to \cite{Ka83}, a JB$^*$-triple is a complex Banach space $X$ admitting a continuous triple product $\J ... :
X\times X\times X \to X,$ which is symmetric and linear in the outer variables, conjugate linear in the middle one,
and satisfies the following axioms:
\begin{enumerate}[{\rm (a)}] \item $L(a,b) L(x,y) = L(x,y) L(a,b) + L(L(a,b)x,y)
	- L(x,L(b,a)y),$  for all $a,b,x,y$ in $X$, 
	where $L(a,b)$ is the operator on $X$ given by $L(a,b) x = \J abx;$
\item For all $a\in X$, $L(a,a)$ is a hermitian operator with non-negative
	spectrum; 
\item $\|\{a,a,a\}\| = \|a\|^3$, for all $a\in X$.\end{enumerate}

In order to provide some examples, let us consider two complex Hilbert spaces $H_1$ and $H_2$, and let $B(H_1,H_2)$ denote the Banach space of all bounded linear operators from $H_1$ into $H_2$. 
The space $B(H_1,H_2)$ is a JB$^*$-triple with respect to the triple product defined by \hyphenation{product}\begin{equation}\label{eq product operators} \J xyz =\frac12 (x y^* z +z y^*
x).\end{equation} In particular, C$^*$-algebras are JB$^*$-triples when equipped with the above triple product.
The Jordan structures enlarge the class of JB$^*$-triples if we consider, for instance, the JB$^*$-algebras in the sense employed in \cite[\S 3.8]{HOS} under the triple product $$\J xyz = (x\circ y^*) \circ z + (z\circ y^*)\circ x -
(x\circ z)\circ y^*.$$\smallskip

Some basic facts and known results about JB$^*$-triples will be needed in the development of this paper. Kaup's \emph{Banach-Stone theorem} states that a linear bijection between JB$^*$-triples is an isometry if and only if it is a triple isomorphism  (cf. \cite[Proposition 5.5]{Ka83}).\smallskip

Let $X$ be a JB$^*$-triple. An element $e$ in $X$ is said to be a \emph{tripotent} if $\J eee=e$. In particular, the partial isometries of a C$^*$-algebra $A$ are precisely its tripotent elements if $A$ is regarded as a JB$^*$-triple respect to the triple product in \eqref{eq product operators}. 
For each tripotent $e\in X$, there exists an algebraic decomposition of $X,$ known as the \emph{Peirce decomposition}\label{eq Peirce decomposition} associated with $e$, which involves the eigenspaces of the operator $L(e,e)$. Namely, $$X= X_{2} (e) \oplus X_{1} (e) \oplus X_0 (e),$$ where $X_i (e)=\{ x\in X : \J eex = \frac{i}{2}x \}$
for each $i=0,1,2$. It is easy to see that every Peirce subspace $X_i(e)$ is a JB$^*$-subtriple of $X$.\smallskip

The so-called Peirce arithmetic assures that $\J {X_{i}(e)}{X_{j} (e)}{X_{k} (e)}\subseteq X_{i-j+k} (e)$ if $i-j+k \in \{ 0,1,2\},$ and $\J {X_{i}(e)}{X_{j} (e)}{X_{k} (e)}=\{0\}$ otherwise, and $$\J {X_{2} (e)}{X_{0}(e)}{X} = \J {X_{0} (e)}{X_{2}(e)}{X} =0.$$
The projection $P_{k_{}}(e)$ of $X$ onto $X_{k} (e)$ is called the Peirce $k$-projection. It is known that Peirce projections are contractive (cf. \cite[Corollary 1.2]{FriRu85}) and satisfy that $P_{2}(e) = Q(e)^2,$ $P_{1}(e) =2(L(e,e)-Q(e)^2),$ and $P_{0}(e) =\hbox{Id}_E - 2 L(e,e) + Q(e)^2,$ where $Q(e):X\to X$ is the conjugate linear map defined by $Q(e) (x) =\{e,x,e\}$. A tripotent $e$ in $X$ is called \emph{unitary} (respectively, \emph{complete} or \emph{maximal}) if $X_2 (e) = X$ (respectively, $X_0 (e) =\{0\}$). 
\smallskip

It is worth remarking that the Peirce space $X_2 (e)$ is a unital JB$^*$-algebra with unit $e$,
product $x\circ_e y := \J xey$ and involution $x^{*_e} := \J exe$, respectively. Actually, Kaup's Banach-Stone theorem \cite[Proposition 5.5]{Ka83} implies that the triple product in $X_2 (e)$ is uniquely determined by the identity $$\{ a,b,c\} = (a \circ_e b^{*_e}) \circ_e c + (c\circ_e b^{*_e}) \circ_e a - (a\circ_e c) \circ_e b^{*_e}, \ \ \ \ (\forall a,b,c\in X_2 (e)).$$

Elements $a,b$ in a JB$^*$-triple $X$ are said to be \emph{orthogonal} (written $a\perp b$) if $L(a,b) =0$. It is known that $a\perp b$ $\Leftrightarrow$ $\J aab =0$ $\Leftrightarrow$ $\{b,b,a\}=0$ $\Leftrightarrow$ $b\perp a$. Let $e$ be a tripotent in $X$. It follows from Peirce arithmetic that $a\perp b$ for every $a\in X_2(e)$ and every $b\in X_0(e)$.
Let $e$ and $u$ be tripotents in $X$, then $$ u\perp e \Leftrightarrow u\pm e \hbox{ are tripotents}$$ (cf. \cite[Lemma 3.6]{IsKaRo95}).
\smallskip

We shall consider the following natural partial order\label{eq partial order tripotents} on the set $\mathcal{U} (X)$, of all tripotents in a JB$^*$-triple $X$, defined by $u \leq e$ if $e-u$ is a tripotent in $X$ with $e-u \perp u$.\smallskip

Complete tripotents play a fundamental role in the extremal structure of the closed unit ball of a JB$^*$-triple $X$. Indeed, the extreme points of the closed unit ball of $X$ coincide with the complete tripotents in $X$ (cf. \cite[Lemma 4.1]{BraKaUp78} and \cite[Proposition 3.5]{KaUp77}).\smallskip

A JBW$^*$-triple is a JB$^*$-triple which is also a dual Banach space (with a unique isometric predual \cite{BarTi86}). It is known that the second dual of a JB$^*$-triple is a JBW$^*$-triple (compare \cite{Di86}). An extension of Sakai's theorem assures that the triple product of every JBW$^*$-triple is separately weak$^*$-continuous (cf. \cite{BarTi86} or \cite{Horn87}).\smallskip

As we commented before, throughout this paper we shall exhibit some new spaces satisfying a Krein--Milman type theorem. The starting point is the celebrated Russo-Dye theorem (see \cite{RuDye}). This result naturally involves the concept of \emph{unitary} element. Let $A$ be a unital C$^*$-algebra. An element $u\in A$ is a \emph{unitary} if it is invertible with $u^{-1} =u^*$, i.e., $u u^* = u^* u ={1}$. Similarly, an element $u$ in a unital JB$^*$-algebra $B$ is called unitary if $u$ is Jordan invertible in $B$ and its (unique) Jordan inverse in $B$ coincides with $u^*$ (compare \cite[\S 3.2]{HOS}). 
Every unital C$^*$-algebra $A$ can be regarded as a unital JB$^*$-algebra equipped with the Jordan product given by $a\circ b := \frac12 (a b + ba)$, and every JB$^*$-algebra is included in the class of JB$^*$ triples. Fortunately, the three definitions of unitary elements given in previous paragraphs coincide for elements in $A$.\smallskip

Finally, we shall make a brief incursion into the theory of real JB$^*$-triples. A \emph{real JB$^*$-triple} is by definition a real closed subtriple of a JB$^*$-triple (see \cite{IsKaRo95}). Every JB$^*$-triple is a real JB$^*$-triple when it is regarded as a real Banach space. As in the case of real C$^*$-algebras, real JB$^*$-triples can be obtained as \emph{real forms} of JB$^*$-triples. More concretely, given a real JB$^*$-triple $E$, there exists a unique (complex) JB$^*$-triple structure on its algebraic complexification $X= E \oplus i E,$ and a conjugation (i.e. a conjugate linear isometry of period 2) $\tau$ on $X$ such that $$E = X^{\tau} = \{ z\in X : \tau (z) = z\},$$ (see \cite{IsKaRo95}). Consequently, every real C$^*$-algebra is a real JB$^*$-triple with respect to the product given in \eqref{eq product operators}, and the Banach space $B(H_1,H_2)$ of all bounded real linear operators between two real, complex, or quaternionic Hilbert spaces also is a real JB$^*$-triple with the same triple product.\smallskip

As in the complex case, an element $e$ in a real JB$^*$-triple ${E}$ is said to be a \emph{tripotent} if $\{e,e,e\}=e$. We shall also write $\mathcal{U}(E)$ for the set of all tripotents in $E$. It is known that an element $e\in E$ is a tripotent in $E$ if and only if it is a tripotent in the complexification of $E$. Each tripotent $e$ in $E$ induces a Peirce decomposition of $E$ in similar terms to those we commented in page \pageref{eq Peirce decomposition} with the exception that $E_2(e)$ is not, in general, a JB$^*$-algebra but a real JB$^*$-algebra (i.e. a closed $^*$-invariant real subalgebra of a (complex) JB$^*$-algebra). Unitary and complete tripotents are defined analogously to the complex setting. Furthermore, the extreme points of $\mathcal{B}_{E}$ coincide with the complete tripotents in the real JB$^*$-triple $E$ (cf. \cite[Lemma 3.3]{IsKaRo95}).\smallskip

Along this note, given a convex set $L$ we denote by $\partial_{e} (L)$ the set of all extreme points in $L$.

\section{Hilbert $C(K)$-modules whose unit spheres are not isometrically isomorphic to the unit sphere of a C$^*$-algebra}\label{sec:2}

One of the aims of this paper is to exhibit the usefulness of a good knowledge on real linear isometries between JB$^*$-triples to study the Mazur--Ulam property on new classes of Banach spaces of continuous functions. We should convince the reader that the recent outstanding achievements obtained by Mori and Ozawa for unital C$^*$-algebras in \cite{MoriOza2018} are not enough to conclude that some natural spaces of vector-valued continuous functions satisfy the Mazur--Ulam property.\smallskip

Suppose $H$ is a complex Hilbert space whose inner product is denoted by $\langle . | .\rangle$, and let $K$ be a compact Hausdorff space. It is clear that $C(K,H)$ is a $C(K)$-bimodule with the product defined by $(a f) (t)=(f a) (t) = f(t) a(t)$ for all $t\in K$, $a\in C(K,H)$ and $f\in C(K)$. We consider a sesquilinear $C(K)$-valued mapping on $C(K,H)$ given by the following assignment $$ \langle . | .\rangle : C(K,H)\times C(K,H)\to C(K),  \ \ \langle a | b \rangle (t) := \langle a(t) | b(t)\rangle\ \ (t\in K, a,b\in C(K,H)).$$ It is easy to check that this sesquilinear mapping satisfies the following properties:\begin{enumerate}[$(1)$]\item $\langle a | b\rangle = \langle b | a\rangle^*$;
	\item $\langle f a | b\rangle = f \langle a | b\rangle;$
	\item $\langle a | a\rangle\geq 0$ and $\langle a | a\rangle =0$ if and only if $a=0$,
\end{enumerate}
for all $a,b\in C(K,H)$, $f\in C(K)$. We can therefore conclude that $C(K,H)$ is a Hilbert $C(K)$-module in the sense introduced by I. Kaplansky in \cite{Kapl53}, and consequently, $C(K,H)$ is a JB$^*$-triple with respect to the triple product defined by \begin{equation}\label{eq triple product Hilbert C*-module}  \{a,b,c \}=  \frac12 \langle a | b \rangle c + \frac12 \langle c | b \rangle a, \ \ (a,b,c\in C(K,H)),
\end{equation} (see \cite[Theorem 1.4]{Isid2003}). By a little abuse of notation, the symbol $\langle\cdot |\cdot \rangle$ will indistinctly stand for the inner product of $H$ and the $C(K)$-valued inner product of $C(K,H)$. \smallskip

Throughout this note $\KK$ will stand for $\mathbb{R}$ or $\mathbb{C}$. Given $\eta\in H$ and a mapping $f:K \to \KK$, the symbol $\eta\otimes f$ will denote the mapping from $K$ to $H$ defined by $\eta\otimes f (t) = f(t) \eta$ ($t\in K$). We note that $\eta\otimes f$ is continuous whenever $f\in C(K)$. We will use the juxtaposition for the pointwise product between maps whenever such a product makes sense.\smallskip

Let us consider vector-valued continuous functions on a compact Hausdorff space $K$ with values in a Banach space $X$. 
It is known that if $e\in \partial_e (\mathcal{B}_{C(K,X)})$, then $\|e(t)\| = 1$ (that is, $e(t)\in S(X)$) for all $t \in K$ (cf. \cite[Lemma 1.4]{AronLohm}). The reciprocal implication is not true in general, however, if $X$ is a strictly convex Banach space, then we have \begin{equation}\label{eq extreme points conts functions} \hbox{$e\in \partial_e (\mathcal{B}_{C(K,X)})$ if and only if $\|e(t)\| = 1$ for all $t \in K$},
\end{equation} (cf. \cite[Remark 1.5]{AronLohm}).\smallskip

\begin{theorem}\label{t C(K,H) cannot be isometrically isomorphic to the sphere of a C*-algebra} Let $K$ be a compact Hausdorff space, and let $H$ be a complex Hilbert space with dimension bigger than or equal to 2. Then there exists no surjective isometry from the unit sphere of $C(K,H)$ onto the unit sphere of a C$^*$-algebra.
\end{theorem}

\begin{proof} Arguing by contradiction we assume the existence of a C$^*$-algebra $A$ and a surjective isometry $\Delta :  S(A)\to S(C(K,H)).$ Since $A$ and $C(K,H)$ are JB$^*$-triples, it follows from \cite[Corollary 2.5$(b)$ and comments prior to it]{FerGarPeVill17} that $\Delta (\partial_e(\mathcal{B}_{A})) = \partial_e(\mathcal{B}_{C(K,H)})$. The non-emptiness of the set  $\partial_e(\mathcal{B}_{C(K,H)})$ assures that $\partial_e(\mathcal{B}_{A})\neq \emptyset$. It is well known that in such case $A$ must be unital (cf. \cite[Proposition 1.6.1]{S}). A recent result by Mori and Ozawa shows that every unital C$^*$-algebra satisfies the Mazur--Ulam property (see \cite[Theorem 1]{MoriOza2018}). Therefore $\Delta$ extends to a surjective real linear isometry $T : A\to C(K,H)$.\smallskip
	
	Now, since $T : A\to C(K,H)$ is a surjective real linear isometry, $A$ is a C$^*$-algebra and $C(K,H)$ is a JB$^*$-triple, we can apply \cite[Theorem 3.1]{Da} or \cite[Theorem 3.2 and Corollary 3.4]{FerMarPe} (see also \cite[Theorem 3.1]{FerPe17d}) to deduce that $T$ is a triple isomorphism when $A$ and $C(K,H)$ are equipped with the triple products given in \eqref{eq product operators} and \eqref{eq triple product Hilbert C*-module}, respectively. Let $\textbf{1}$ denote the unit element in $A$. Clearly, $A_2(\textbf{1}) =A$. Since $\Delta(\textbf{1})$ must be a unitary in $C(K,H)$, in particular $\Delta(\textbf{1}) (t) \in S(H)$ for every $t\in K$ (cf. \eqref{eq extreme points conts functions}). Let us fix $t_0\in K$. By applying that dim$(H)\geq 2$, we can find $\eta\in S(H)$ satisfying $\langle\eta | \Delta(\textbf{1}) (t_0)\rangle =0$. We consider the element $a = \eta\otimes 1$, where $1$ is the unit element in $C(K)$. In this case $\{ \Delta(\textbf{1}), \Delta(\textbf{1}) , a\} (t_0) = \frac12 \eta\neq a(t_0),$ and thus $a\notin C(K,H)_2 (\Delta(\textbf{1}))$. 
\end{proof}

\smallskip

Let us observe another point of view to deal with $C(K,H)$ as a real JB$^*$-triple. Indeed, suppose $H$ is a complex Hilbert space with inner product $\langle . | .\rangle$. We can regard $H$ as a real Hilbert space with its underlying real space and the inner product defined by $(a|b)=\Re\hbox{e}\langle a | b\rangle$ ($a,b\in H$), the latter real Hilbert space will be denoted by $H_{_{\mathbb{R}}}$. In general, the inner product of a real Hilbert space will be denoted by $(a|b)$. Let $K$ be a compact Hausdorff space. Let us observe that the norms of $C(K,H)$ and $C(K,H_{_{\mathbb{R}}})$ both coincide. We can therefore reduce to the case in which $H$ is a real Hilbert space. We shall always consider $C(K,H)$ as a real JB$^*$-triple with respect to the triple product $$\J abc :=\frac12 (a|b) c +\frac12 (c|b) a.$$

For each $x_0$ in $H$, we shall write $x_0^*$ for the unique functional in $H^*$ defined by $x_0^* (x) = \langle x | x_0\rangle$ ($x\in H$). Given $t_0\in K$, $\delta_{t_0}: C(K,H)\to H$ will stand for the bounded linear operator defined by $\delta_{t_0}(a)=a(t_0)$ ($a\in C(K,H)$). Finally, let $x_0^*\otimes\delta_{t_0}$ denote the functional on $C(K,H)$ given by $(x_0^*\otimes\delta_{t_0}) (a) := x_0^* (a(t_0))$, for each $a\in C(K,H)$.

\begin{theorem}\label{t C(K,H) cannot be isometrically isomorphic to the sphere of a real W*-algebra} Let $K$ be a compact Hausdorff space, and let $\mathcal{H}$ be a real Hilbert space with dim$(\mathcal{H})=3$ or  dim$(\mathcal{H})\geq 5$. Then there exists no surjective isometry from the unit sphere of $C(K,\mathcal{H})$ onto the unit sphere of a real von Neumann algebra.
\end{theorem}

\begin{proof} It is know that $\partial_{e} (\mathcal{B}_{C(K,\mathcal{H})^*}) =\{ z^*\otimes \delta_{t} : t\in K, \ z\in S(\mathcal{H})\}$ (see \cite[Lemma 1.7 in page 197]{Singer}). It is easy to check that the norm-closed linear span of $\partial_{e} (\mathcal{B}_{C(K,\mathcal{H})^*})$ in $C(K,\mathcal{H})^*$ is precisely the space $\displaystyle \bigoplus^{\ell_1}_{t\in K} \mathcal{H} $. In particular, the atomic part of the real JBW$^*$-triple $C(K,\mathcal{H})^{**},$ in the sense employed and studied in \cite{PeSta2001} and \cite{FerMarPe}, coincides with the direct sum $\displaystyle \bigoplus^{\ell_{\infty}}_{t\in K} \mathcal{H}.$ In other words, every real or complex Cartan factor in the atomic part of $C(K,\mathcal{H})^{**}$ coincides with the real Hilbert space $\mathcal{H}$ equipped with the triple product $\{a,b,c\} :=\frac12 (a|b) c +\frac12 (c|b) a.$\smallskip
	
We shall argue by contradiction. Suppose $A$ is a real von Neumann algebra and $\Delta : S(A)\to  S(C(K,\mathcal{H}))$ is a surjective isometry. Applying \cite[Theorem 1$(2)$]{MoriOza2018} we deduce the existence of a surjective real linear $T : A\to C(K,\mathcal{H})$. The bitransposed mapping $T^{**} : A^{**}\to C(K,\mathcal{H})^{**}$ also is a surjective real linear isometry. It is known that the atomic part of $A^{**}$ coincides with a direct sum of the form $\displaystyle \bigoplus^{\ell_{\infty}}_{\alpha\in \Lambda} B(\mathcal{H}_{\alpha}),$ where each $\mathcal{H}_{\alpha}$ is a Hilbert space over $\mathbb{R},$ $\mathbb{C},$ or $\mathbb{H}$ (see \cite[Lemma 6.2]{ChuDaRuVen} or \cite[\S 5.3]{Li2003}). Arguing as in the proof of \cite[Theorem 3.2]{FerMarPe} we deduce that $T^{**}$ maps the atomic part of $A^{**}$ onto the atomic part of $C(K,\mathcal{H})^{**}$, furthermore, each factor in the atomic part of $A^{**}$ is isometrically mapped by $T^{**}$ onto a factor in the atomic part of $C(K,\mathcal{H})^{**}$. That is, for each $\alpha\in \Lambda$ the restriction $T^{**}|_{B(\mathcal{H}_{\alpha})}: B(\mathcal{H}_{\alpha})\to \mathcal{H}$ is a surjective isometry. Since in $\mathcal{H},$ equipped with the product $\{a,b,c\} :=\frac12 (a|b) c +\frac12 (c|b) a,$ the rank is one (i.e. we cannot find two non-zero orthogonal tripotents), and $T^{**}|_{B(\mathcal{H}_{\alpha})}$ preserves orthogonal tripotents (see \cite[Theorem 4.8]{IsKaRo95} or \cite[Proposition 2.9]{FerMarPe}) it follows that $\mathcal{H}_{\alpha}$ must be a one dimensional Hilbert space over $\mathbb{R},$ $\mathbb{C},$ or $\mathbb{H}$, which is impossible because $T^{**}|_{B(\mathcal{H}_{\alpha})}: B(\mathcal{H}_{\alpha})\to \mathcal{H}$ is a surjective real linear isometry and dim$(\mathcal{H})\in \{ 3 \}\cup \{4 + n: n\in \mathbb{N}\}$.
\end{proof}


\section{Facial structure of real JB$^*$-triples revisited}\label{sec: facial real}

This section is devoted to explore the facial structure of the closed unit ball of a real JB$^*$-triple and its dual space. It is an interesting question by its own right, and moreover, its application will be crucial later in the study of the Mazur--Ulam property.\smallskip

The facial structure of the closed unit ball of a JBW$^*$-triple and its predual was completely determined by C.M. Edwards and G.T. R\"{u}ttimann in \cite{EdRutt88}. In order to review the results, we shall recall some terminology. Let $X$ be a real or complex Banach space with dual space $X^*$. Suppose $F$ and $G$ are two subsets of $\mathcal{B}_{X}$ and $\mathcal{B}_{X^*}$, respectively.  Then we set
$$ F^{\prime} =F^{\prime,X^*} =  \{a \in \mathcal{B}_{X^*}:a(x) = 1\,\, \forall x \in F\},$$
$$G_{\prime} =G_{\prime,X}= \{x \in \mathcal{B}_{X}:a(x) = 1\,\, \forall a \in G\}.$$
Clearly, $F^{\prime}$ is a weak$^*$-closed face of $\mathcal{B}_{X^*}$ and $G_{\prime}$
is a norm-closed face of $\mathcal{B}_{X}$. We say that $F$ is a \emph{norm-semi-exposed face} of $\mathcal{B}_{X}$ (respectively, $G$ is a \emph{weak$^*$-semi-exposed face} of $\mathcal{B}_{X^*}$) if $F=(F^{\prime})_{\prime}$ (respectively, $G= (G_{\prime})^{\prime}$). It is known that the mappings $F \mapsto F^{\prime}$
and $G \mapsto G_{\prime}$ are anti-order isomorphisms between the
complete lattices $\mathcal{S}_n(\mathcal{B}_{X})$, of norm-semi-exposed faces
of $\mathcal{B}_X,$ and $\mathcal{S}_{w^*}( \mathcal{B}_{X^*}),$ of weak$^*$-semi-exposed
faces of $\mathcal{B}_{X^*}$, and are inverses of each other.\smallskip

Recall that a partially ordered set $\mathcal{P}$ is called a \emph{complete lattice}, if, for
any subsets of $\mathcal{S}\subseteq \mathcal{P}$, the supremum and the infimum of $\mathcal{S}$ exist in $\mathcal{P}$. It is shown in \cite[Corollary 4.3]{EdRutt88} that, for each JBW$^*$-triple $M$, the set $\tilde{\mathcal U}(M),$ of all tripotents in $M$ with a largest element adjoined, is a complete lattice with respect to the ordering defined in page \pageref{eq partial order tripotents}.\smallskip

Let $M$ be a JBW$^*$-triple. The main achievements in \cite{EdRutt88} prove that every weak$^*$-closed
face of $\mathcal{B}_{M}$ is weak$^*$-semi-exposed; furthermore, the mapping \begin{equation}\label{eq anti order weak* closed faces complex} u \mapsto
(\{u\}_{\prime})^{\prime} = u + \mathcal{B}_{M_0(u)}
\end{equation} is an anti-order isomorphism from the complete lattice
$\tilde{\mathcal U}(M)$ onto the complete lattice $\mathcal
F_{w^*}(\mathcal{B}_{M})$ of weak$^*$-closed faces of $\mathcal{B}_{M}$ (cf. \cite[Theorem 4.6]{EdRutt88}). Concerning the facial structure of $M_*$,
the same authors proved in \cite[Theorem 4.4]{EdRutt88} that every norm-closed face of $\mathcal{B}_{M_{*}}$ is norm-semi-exposed, and the mapping
\begin{equation}\label{eq order isomorphism norm-closed faces predual complex} u \mapsto \{u\}_{\prime}
\end{equation} is an order isomorphism from $\tilde{\mathcal
U}(M)$ onto the complete lattice $\mathcal F_{n}(\mathcal{B}_{M_{*}})$ of norm-closed faces of $\mathcal{B}_{M_{*}}$.\smallskip

In 1992, C.A. Akemann and G.K. Pedersen studied the norm-closed faces of the closed unit ball of a C$^*$-algebra $A$ and the weak$^*$-closed faces of $\mathcal{B}_{A^*}$ (see \cite{AkPed92}). We had to wait until 2010 to have a description of the norm-closed faces of the closed unit ball of a JB$^*$-triple (\cite{EdFerHosPe2010}). A JB$^*$-triple $X$ might contain no non-trivial tripotents, while the set of all tripotents in $X^{**}$ is too big to be in one-to-one correspondence with the set of norm-closed faces of $\mathcal{B}_{X}$.  The appropriate set is the set of all compact tripotents in $X^{**}$. We continue refreshing the notion of compactness.\smallskip

Let $a$ be a norm-one element in a JB$^*$-triple $X$, and let $X_a$ denote the JB$^*$-subtriple generated by $a$, that is, the closed subspace generated by all odd powers $a^{[2n+1]}$, where $a^{[1]}=a,$ $a^{[3]}=\{a,a,a\}$, and $a^{[2n+1]}=\{a,a,a^{[2n-1]}\}$ ($n\geq 2$). It is known that there exists an isometric triple isomorphism $\Psi : X_a\to C_{0} (L)$ satisfying $\Psi (a) (s) =s,$ for all $s$ in
$L$ (compare \cite[1.15]{Ka83}), where $C_{0} (L)$ is the abelian C$^*$-algebra of all complex-valued
continuous functions on $L$ vanishing at $0$, $L$ being a locally compact subset of $(0,\|a\|]$ satisfying that $\|a\|\in L\cup\{0\}$ is
compact. If $f: L\cup\{0\} \to \mathbb{C}$ is a continuous function vanishing at 0, the \emph{triple functional calculus} of $f$ at the element $a$ is the unique element $f_t(a)\in X_a,$ defined by $f_t(a)=\Psi^{-1}(f)$. We can define this way $a^{[\frac{1}{2n+1}]} := (r_{n}){_t}(a),$ where $r_n(s) = s^{\frac{1}{2n+1}}$ ($s\in L$) and $n\in \mathbb{N}$.\smallskip

When $X$ is regarded as a JB$^*$-subtriple of $X^{**}$, the triple functional calculus $f \mapsto f_t(a)$ admits an extension, denoted
by the same symbol, from $C_0(L)$ to the commutative W$^*$-algebra $W$ generated by $C_0(L)$, onto the JBW$^*$-subtriple $X_a^{**}$ of $X^{**}$
generated by $a$. Observe that the sequences $(a^{[\frac{1}{2n-1}]})_n$ and $(a^{[2n-1]})_n$  in $C_0(L)$ converge in the weak$^*$-topology 
of $C_0(L)^{**}$ to the characteristic functions $\chi_{L}$ and $\chi_{\{1\}}$ of the sets $L$ and $\{1\}$, respectively. The corresponding limits define two tripotents in $X_a^{**}$ which are called the {\it range tripotent} and the {\it support tripotent} of $a$, respectively. These tripotents will be denoted by $r(a)$ and $u(a),$ respectively.\smallskip

For each functional $\varphi$ in the predual, $M_*$, of a JBW$^*$-triple $M$ there exists a unique tripotent $s(\varphi)$ (called the \emph{support tripotent} of $\varphi$) such that $\varphi = \varphi P_2(s(\varphi))$ and $\varphi|_{M_2(s(\varphi))}$ is a faithful normal positive functional on the JBW$^*$-algebra $M_2(s(\varphi))$ (cf. \cite[Proposition 2]{FriRu85}).\smallskip

We are interested in a special property satisfied by the support tripotent. Suppose $a$ is a norm-one element in a JB$^*$-triple $X$. Since $a = u(a) + (a-u(a))$ with $u(a)\perp (a-u(a))$ in $X^{**}$, it follows from \cite[Proposition 1]{FriRu85} that $\{u(a)\}_{\prime,X^*} \subseteq \{a\}^{\prime,X^*}.$ However, if $\phi\in X^*$ satisfies $\|\phi\| = 1= \phi(a)$, we deduce from the definition of the support tripotent of $\phi$ in $X^{**}$ that $P_2(s(\phi)) (a) = s(\phi)$, and hence $a = s(\phi) + P_0(s(\phi)) (a)$ in $X^{**}$ (cf. \cite[Lemma 1.6]{FriRu85}). We therefore conclude that $u(a)\geq s(\phi)$ in $X^{**}$, and thus $\phi (u(a)) = 1$, witnessing that $\{u(a)\}_{\prime,X^*} = \{a\}^{\prime,X^*}$ and consequently, \begin{equation}\label{eq subprime prime with support trip} \left(\{a\}^{\prime,X^*}\right)^{\prime,X^{**}} = \left(\{u(a)\}_{\prime,X^*}\right)^{\prime,X^{**}}.
\end{equation}

A tripotent $u$ in the JBW$^*$-triple $X^{**}$ is said to be \emph{compact-$G_{\delta}$} if $u$ coincides with the support tripotent of a norm-one element in $X$. The tripotent $u$ is said to be \emph{compact} if $u=0$ or there exists a decreasing net of compact-$G_{\delta}$ tripotents in $X^{**}$ 
whose infimum is $u$ (compare \cite[\S 4]{EdRu96}). Henceforth we shall write $\tilde{\mathcal U}_{c} (X^{**})$ for the set of all compact tripotents in $X^{**}$ with a largest element adjoined. Having these notions in mind we can understand the main result in \cite{EdFerHosPe2010}: Every norm-closed face of $\mathcal{B}_{X}$ is norm-semi-exposed and the mapping 
$$u \mapsto (\{u\}_{\prime})_{\prime} = (u + \mathcal{B}_{X^{**}_0(u)})\cap X$$
is an anti-order isomorphism from $\tilde{\mathcal U}_c(X^{**})$ onto the complete lattice $\mathcal{F}_{n}(\mathcal{B}_{X})$ of norm-closed faces of $\mathcal{B}_{X}$ (cf. \cite[Corollaries 3.11 and 3.12]{EdFerHosPe2010}). The study is completed in \cite{FerPe10}, where it is shown that the mapping $$u \mapsto
\{u\}_{\prime}$$ is an order isomorphism from $\tilde{\mathcal U}_c(X^{**})$ onto the complete lattice $\mathcal F_{w^*}(\mathcal{B}_{X^*})$ of weak$^*$-closed faces of $\mathcal{B}_{X^*}$. \smallskip

In the setting of real JBW$^*$-triples, C.M. Edwards and G.T. R\"{u}ttimann proved in \cite{EdRutt2001} that the conclusions in \eqref{eq anti order weak* closed faces complex} and \eqref{eq order isomorphism norm-closed faces predual complex} holds when $M$ is a real JBW$^*$-triple. However, as long as we know, the facial structure of the closed unit ball of a real JB$^*$-triple remains unexplored. We shall try to fill this gap. \smallskip

We begin with a very basic result. Let us consider a complex Banach space $X$ equipped with a conjugation $\tau:X\to X$ (a conjugate linear isometry of period 2), and set $E=X^\tau=\{x\in X : \tau(x)=x\}$. The mapping $P:X\to X$ defined by $P(x)=\frac12 (x+\tau(x))$ ($x\in X$), is a contractive real linear projection whose image is $E$. The mapping $\tau^{\sharp} : X^*\to X^*$, $\tau^{\sharp} (\varphi) (x) :=\overline{\varphi(\tau(x))}$ ($x\in X$, $\varphi\in X^*$) is a conjugation on $X^*$, and the correspondence $\varphi\mapsto \varphi|_{E}$ defines a surjective real linear isometry from $(X^*)^{\tau^{\sharp}}$ onto $E^*$. We can similarly define a conjugation $\tau^{\sharp\sharp}$ on $X^{**}$ satisfying that $(X^{**})^{\tau^{\sharp\sharp}}$ is isometrically isomorphic to $E^{**}$. In particular, the weak$^*$-topology of $E^{**}$ coincides with the restriction to $E^{**}$ of the weak$^*$-topology of $X^{**}$. Clearly, if a functional $\varphi$ in $X^*$ is a $\tau^{\sharp}$-symmetric (equivalently, $\varphi\in E^*$), its support tripotent in $X^{**}$ is $\tau^{\sharp\sharp}$-symmetric and hence lies in $E^{**}$.\smallskip

Let $F$ be a subset of $\mathcal{B}_E$. We set $\mathfrak{F}:=P^{-1}(F)\cap \mathcal{B}_X$. It is standard to check that \begin{equation}\label{eq lifting faces} F\in \mathcal{F}_{n}(\mathcal{B}_{E}) \Leftrightarrow \mathfrak{F}\in \mathcal{F}_{n}(\mathcal{B}_X).
\end{equation}

Henceforth we assume that $X$ is a complex JB$^*$-triple, and thus $E$ is a real JB$^*$-triple. Proposition 5.5 in \cite{Ka83} assures that $\tau$ is a conjugate linear triple automorphism. It is not hard to see that ${\mathcal{U}} (E) = {\mathcal{U}} (X)^{\tau} =\{e\in {\mathcal{U}} (X) : \tau (e) =e\}$, and what is even more interesting ${\mathcal{U}} (E^{**}) = {\mathcal{U}} (X^{**})^{\tau^{\sharp\sharp}} =\{e\in {\mathcal{U}} (X^{**}) : \tau^{\sharp\sharp} (e) =e\}$. It follows from \cite[Lemma 3.4$(ii)$]{EdRutt2001} that the set $\tilde{\mathcal{U}} (E^{**})$ of all tripotents in $E^{**}$ with a largest element adjoined is a sub-complete lattice of $\tilde{\mathcal{U}} (X^{**})$.\smallskip

If $a$ is a norm-one element in $E$ (that is, an element in $X$ with $\tau(a) =a$). Since $\tau (a^{[\frac{1}{2n-1}]})= \tau (a)^{[\frac{1}{2n-1}]} = a^{[\frac{1}{2n-1}]}$ and $\tau (a^{[2n-1]})= a^{[2n-1]},$ for all natural $n$, $E^{**}$ is weak$^*$-closed in $X^{**},$ and $\tau^{\sharp\sharp}$ is weak$^*$-continuous, we deduce that $\tau^{\sharp\sharp} (r(a)) =r(a)$ and $\tau^{\sharp\sharp}(u(a)) = u(a)$, that is, the range and support tripotents of $a$ in $X^{**}$ are $\tau^{\sharp\sharp}$-symmetric elements in $X^{**}$, and thus they both are tripotents in $E^{**}$, called \emph{range} and \emph{support} tripotents of $a$ in $E^{**}$. Combining \eqref{eq subprime prime with support trip} with the previous conclusions we get \begin{equation}\label{eq subprime prime with support trip real} \{a\}^{\prime,E^*}= \{u(a)\}_{\prime,E^*}, \hbox{ and } \left(\{a\}^{\prime,E^*}\right)^{\prime,E^{**}} = \left(\{u(a)\}_{\prime,E^*}\right)^{\prime,E^{**}}.
\end{equation}

Thanks to the above facts, the notion of compact tripotent fits well in the setting of real JB$^*$-triples. A tripotent $u$ in $E^{**}$ will be called \emph{compact-$G_{\delta}$} if $u$ coincides with the support tripotent of a norm-one element in $E$. The tripotent $u$ is called \emph{compact} if $u=0$ or there exists a decreasing net of compact-$G_{\delta}$ tripotents in $E^{**}$ 
whose infimum is $u$. As in the complex setting, we shall write $\tilde{\mathcal U}_c(E^{**})$ for the set of all compact tripotents in $E^{**}$ with a largest element adjoined.\smallskip

It is absolutely clear that every compact(-$G_{\delta}$) tripotent in $E^{**}$ is a $\tau^{\sharp\sharp}$-symmetric compact(-$G_{\delta}$) tripotent in $X^{**}$. The reciprocal is not obvious. To prove it we shall extend a result of Edwards and R\"{u}ttimann which affirms that a tripotent $u\in X^{**}$ is compact if and only if the face $\{u\}_{\prime, X^*}$ is weak$^*$-semi-exposed in $\mathcal{B}_{X^*}$ (cf. \cite[Theorem 4.2]{EdRu96}). We recall first a lemma borrowed from \cite{EdRutt2001}.\smallskip

\begin{lemma}\label{l ER real faces 3.6}\cite[Lemma 3.6]{EdRutt2001} Let $\tau$ be a conjugation on a JB$^*$-triple $X$, and let $E= X^{\tau}$. Then for each tripotent $u\in E^{**} = (X^{**})^{\tau^{\sharp\sharp}}$ we have $$ \{u\}_{\prime,E^*} = \left(\{u\}_{\prime,X^*}\right)^{\tau^{\sharp}} = \{u\}_{\prime,X^*}\cap E^* = \{u\}_{\prime,X^*}.$$
\end{lemma}

We establish next a real version of \cite[Theorem 4.2]{EdRu96}.

\begin{proposition}\label{p characterization of compact tripotents by semi-exposed faces} Let $\tau$ be a conjugation on a JB$^*$-triple $X$, and let $E= X^{\tau}$. A tripotent $u$ in the real JBW$^*$-triple $E^{**}$ is compact if and only if $\{u\}_{\prime,E^*}$ is weak$^*$-semi-exposed in $\mathcal{B}_{E^*}$.
\end{proposition}

\begin{proof} Suppose $u$ is a non-trivial compact tripotent in $E^{**}$. According to what we commented before this proposition, $u$ is a $\tau^{\sharp\sharp}$-symmetric compact tripotent in $X^{**}$. Theorem 4.2  in \cite{EdRu96} implies that $\{u\}_{\prime, X^*}$ is weak$^*$-semi-exposed in $\mathcal{B}_{X^*}$, that is $$ \left(\left(\{u\}_{\prime, X^*}\right)_{\prime, X}\right)^{\prime, X^*} = \{u\}_{\prime, X^*}.$$

It follows from Lemma \ref{l ER real faces 3.6} that
$$\{u\}_{\prime,E^*} = \left(\{u\}_{\prime,X^*}\right)^{\tau^{\sharp}} = \{u\}_{\prime,X^*}\cap E^* = \{u\}_{\prime,X^*}.$$ 

We shall next show that the non-empty set $\left(\{u\}_{\prime, X^*}\right)_{\prime, X}\subseteq S(X)$ is $\tau$-symmetric. Take $x\in \left(\{u\}_{\prime, X^*}\right)_{\prime, X}$ and $\varphi\in \{u\}_{\prime, X^*}= \{u\}_{\prime,E^*}$. Since $\tau^{\sharp}(\varphi)= \varphi\in \{u\}_{\prime, X^*}$ we have $$1= \tau^{\sharp} (\varphi) (x) = \overline{\varphi (\tau(x))} = {\varphi (\tau(x))},$$ witnessing that $\tau(x)\in \left(\{u\}_{\prime, X^*}\right)_{\prime, X}$, and thus $\tau \left(\left(\{u\}_{\prime, X^*}\right)_{\prime, X}\right) = \left(\{u\}_{\prime, X^*}\right)_{\prime, X}$. We have also shown that for each $x\in \left(\{u\}_{\prime, X^*}\right)_{\prime, X}$ and $\varphi\in \{u\}_{\prime, X^*}$ we have $$1= \varphi\left(\frac{x+\tau(x)}{2}\right) \leq \left\|\frac{x+\tau(x)}{2} \right\| \leq 1.$$ It follows from the above that $\left(\{u\}_{\prime, X^*}\right)_{\prime, X} \cap E$ is a non-empty subset of $S(E)$ which coincides with $\left(\{u\}_{\prime, E^*}\right)_{\prime, E}$ and \begin{equation}\label{eq being weak*-semi-exposed in the real part} \left(\left(\{u\}_{\prime, E^*}\right)_{\prime, E}\right)^{\prime, E^*} = \{u\}_{\prime, E^*},
\end{equation} which guarantees that $\{u\}_{\prime,E^*}$ is weak$^*$-semi-exposed in $\mathcal{B}_{E^*}$.\smallskip

Suppose now that $\{u\}_{\prime,E^*}$ is weak$^*$-semi-exposed in $\mathcal{B}_{E^*}$, that is, the equality in \eqref{eq being weak*-semi-exposed in the real part} holds. We can literally follow the arguments contained in the proof of \cite[Theorem 4.2]{EdRu96}. The details are included here for completeness reasons. It follows from the equality in \eqref{eq being weak*-semi-exposed in the real part} that the convex set $\left(\{u\}_{\prime, E^*}\right)_{\prime, E}$ is a non-empty norm-closed face of $\mathcal{B}_E$.
For each $a\in \left(\{u\}_{\prime, E^*}\right)_{\prime, E}$ let face$(a)$ denote the smallest face of $B_{E}$ containing $\{a\}$ and set $\Lambda =\{ \hbox{face}(a) : a\in \left(\{u\}_{\prime, E^*}\right)_{\prime, E}\}.$ Since for each $a_1,a_2\in \left(\{u\}_{\prime, E^*}\right)_{\prime, E},$ both face$(a_1)$ and face$(a_2)$ are contained in face$(\frac{1}{2}(a_1+a_2)),$ we
conclude that $\Lambda$ is a partially ordered by set inclusion which is upward directed. We can further check that if $$a_1\in\hbox{face}(a_1)\subseteq \hbox{face}(a_2)\subseteq \left(\{a_2\}^{\prime,E^*}\right)^{\prime,E^{**}} = \hbox{(by \eqref{eq subprime prime with support trip real})} = \left(\{u(a_2)\}_{\prime,E^*}\right)^{\prime,E^{**}},$$ then $$\left(\{u(a_1)\}_{\prime,E^*}\right)^{\prime,E^{**}}= \left(\{a_1\}^{\prime,E^*}\right)^{\prime,E^{**}} \subseteq \left(\left(\left(\{a_2\}^{\prime,E^*}\right)^{\prime,E^{**}}\right)_{\prime,E^*}\right)^{\prime,E^{**}} $$ $$= \left(\left(\left(\{u(a_2)\}_{\prime,E^*}\right)^{\prime,E^{**}}\right)_{\prime,E^*}\right)^{\prime,E^{**}} =
\left(\{u(a_2)\}_{\prime,E^*}\right)^{\prime,E^{**}}.$$ The description of the weak$^*$-closed faces in $\mathcal{B}_{E^{**}}$ proved in \cite[Theorem 3.9]{EdRutt2001} gives $u(a_1)\geq u(a_2)$.\smallskip

We define a net now. For each $\mu\in \Lambda$ we set $u_{\mu} = u(a)$, where $a\in \left(\{u\}_{\prime, E^*}\right)_{\prime, E}$ satisfies $\mu=\hbox{face}(a).$ We have shown in the previous paragraphs that $\{ u_{\mu}\}_{\mu\in \Lambda}$ is a decreasing net of compact-G$_{\delta}$ in $E^{**}$. In particular, the net $\{ u_{\mu} \}_{\mu\in \Lambda}$ converges in the weak$^*$-topology of $E^{**}$ to its infimum. Let $v$ denote this infimum, which is, by definition, a compact tripotent in $E^{**}$. \smallskip

For each $\mu \in \Lambda$, we have $u_{\mu} = u(a)$, with $a\in \left(\{u\}_{\prime, E^*}\right)_{\prime, E}.$ Therefore $$\left(\{u_{\mu}\}_{\prime,E^*}\right)^{\prime,E^{**}}= \left(\{u(a)\}_{\prime,E^*}\right)^{\prime,E^{**}}= \left(\{a\}^{\prime,E^*}\right)^{\prime,E^{**}} \subseteq \left(\{u\}_{\prime,E^*}\right)^{\prime,E^{**}},$$ which by a new application of \cite[Theorem 3.9]{EdRutt2001}, proves $u\leq u_{\mu}$ for every $\mu \in \Lambda$, and consequently, $u\leq v$.\smallskip

Finally, for each $a\in \left(\{u\}_{\prime, E^*}\right)_{\prime, E}$, we know that $v\leq u(a)= u_{\mu}$ with $\mu = \hbox{face} (a)$, which implies that $\{v\}_{\prime,E^*} \subseteq \{u_{\mu}\}_{\prime,E^*} = \{a\}^{\prime,E^*}$. We deduce from the arbitrariness of $a\in \left(\{u\}_{\prime, E^*}\right)_{\prime, E}$ that $\{v\}_{\prime,E^*} \subseteq \left(\left(\{u\}_{\prime, E^*}\right)_{\prime, E}\right)^{\prime,E^*}= \{u\}_{\prime, E^*},$ where the last equality follows from the hypothesis. Therefore $v\leq u$ (cf. \cite[Theorem 3.7 or 3.9]{EdRutt2001}), witnessing that $u = v$ is a compact tripotent in $E^{**}$.
\end{proof}

We can now prove that compact tripotents in the second dual of a real JB$^*$-triple are compact in the second dual of its complexification.

\begin{corollary}\label{c tripotent second dual is compact iff it is compact in the complexification} Let $\tau$ be a conjugation on a JB$^*$-triple $X$, and let $E= X^{\tau}$. Suppose $u$ is a tripotent in $E^{**}$. Then the following assertions are equivalent:\begin{enumerate}[$(a)$]\item $u$ is compact in $E^{**}$;
\item $u$ is compact in $X^{**}$.
\end{enumerate}
\end{corollary}

\begin{proof} The implication $(a)\Rightarrow (b)$ has been commented before Lemma \ref{l ER real faces 3.6}.\smallskip

$(b)\Rightarrow (a)$ Suppose that $u$ is compact in $X^{**}$. Theorem 4.2 in \cite{EdRu96} assures that $\{u\}_{\prime,X^*}$ is weak$^*$-semi-exposed. Lemma \ref{l ER real faces 3.6} shows that $\{u\}_{\prime,X^*} = \{u\}_{\prime,E^*}$. The arguments in the proof of the ``only if'' implication in Proposition \ref{p characterization of compact tripotents by semi-exposed faces} assure that $\{u\}_{\prime,E^*}$ is weak$^*$-semi-exposed in $\mathcal{B}_{E^*}$. The ``if'' implication of Proposition \ref{p characterization of compact tripotents by semi-exposed faces} proves that $u$ is compact in $E^{**}$.
\end{proof}

In the setting of (complex) JB$^*$-triples a new characterization of compact tripotents in the second dual has been recently established in \cite{BeCuFerPe2018}. The concrete result reads as follows:

\begin{theorem}\label{t characterization relatively open weak* closed faces}{\rm\cite[Theorem 3.6]{BeCuFerPe2018}} Let $X$ be a JB$^*$-triple. Suppose $F$ is a proper weak$^*$-closed face of the closed unit ball of $X^{**}$. Then the following statements are equivalent: \begin{enumerate}[$(a)$]\item $F$ is open relative to $X$, that is, $F\cap X$ is weak$^*$-dense in $F$;
\item $F$ is a weak$^*$-closed face associated with a non-zero compact tripotent in $X^{**}$, that is, there exists a unique non-zero compact tripotent $u$ in $X^{**}$ satisfying $F= u + \mathcal{B}_{X^{**}_0(u)}$.$\hfill\Box$
\end{enumerate}
\end{theorem}

We shall make use of the previous theorem to determine the norm-closed faces of the closed unit ball of a real JB$^*$-triple. 

\begin{theorem}\label{t facial structure norm-closed faces in a real JB-triple} Let $\tau$ be a conjugation on a JB$^*$-triple $X$, and let $E= X^{\tau}$. Then for each norm-closed proper face $F$ of $\mathcal{B}_{E}$ there exists a unique compact tripotent $u\in E^{**}$ satisfying $F = (u+\mathcal{B}_{E^{**}_0(u)})\cap E$. Furthermore, the mapping $$ u \mapsto
(\{u\}_{\prime,E^*})_{\prime,E} = (u + \mathcal{B}_{E^{**}_0 (u)})\cap E$$ is an anti-order isomorphism from $\tilde{\mathcal U}_c(E^{**})$ onto $\mathcal{F}_{n}(\mathcal{B}_{E})$.
\end{theorem}

\begin{proof} Suppose $F$ is a norm-closed proper face of $\mathcal{B}_{E}$. Let $P= \frac{1}{2}(Id_X+\tau)$. Then $P$ is a contractive real linear projection on $X$ whose image is $E$. By \eqref{eq lifting faces}, the set $\mathfrak{F}:=P^{-1}(F)\cap \mathcal{B}_X$ is a norm-closed proper face of $\mathcal{B}_{X}$. It is not hard to check that, since $P(\tau(x))= P(x)$, for all $x\in X$, we have $\tau(\mathfrak{F}) =\mathfrak{F}$. By \cite[Corollary 3.12]{EdFerHosPe2010} there exists a unique compact tripotent $u\in X^{**}$ satisfying $\mathfrak{F} =
(\{u\}_{\prime,X^*})_{\prime,X} = (u + \mathcal{B}_{X^{**}_0 (u)})\cap X$. An application of Theorem \ref{t characterization relatively open weak* closed faces} guarantees that $$\tau^{\sharp\sharp} (u) + \mathcal{B}_{X^{**}_0 (\tau^{\sharp\sharp} (u))} = \tau^{\sharp\sharp} (u + \mathcal{B}_{X^{**}_0 (u)}) = \tau^{\sharp\sharp} (\overline{\mathfrak{F}}^{w^*}) = \overline{\tau(\mathfrak{F})}^{w^*} = \overline{\mathfrak{F}}^{w^*} = u + \mathcal{B}_{X^{**}_0 (u)}.$$ A new application of \cite[Corollary 3.12]{EdFerHosPe2010}, implies that $\tau^{\sharp\sharp} (u) = u\in E^{**}$. Corollary \ref{c tripotent second dual is compact iff it is compact in the complexification} shows that $u$ is compact in $E^{**}$, and it is not hard to check that $F = \mathfrak{F}\cap E = \mathfrak{F}^{\tau} = (u + \mathcal{B}_{E^{**}_0 (u)})\cap E,$ as desired. The rest is clear.

\end{proof}

We can now prove the main goal of this subsection which is a tool required for latter purposes.

\begin{theorem}\label{t characterization relatively open weak* closed faces real case} Let $\tau$ be a conjugation on a JB$^*$-triple $X$, and let $E= X^{\tau}$. Suppose $F$ is a proper weak$^*$-closed face of the closed unit ball of $E^{**}$. Then the following statements are equivalent: \begin{enumerate}[$(a)$]\item $F$ is open relative to $E$, that is, $F\cap E$ is weak$^*$-dense in $F$;
\item $F$ is a weak$^*$-closed face associated with a non-zero compact tripotent in $E^{**}$, that is, there exists a unique non-zero compact tripotent $u$ in $E^{**}$ satisfying $F= F_u^{E^{**}} =u + \mathcal{B}_{E^{**}_0(u)}$.
\end{enumerate}
\end{theorem}

\begin{proof} Let $P= \frac{1}{2}(Id_X+\tau^{\sharp\sharp})$. Then $P$ is a contractive weak$^*$-continuous real linear projection on $X^{**}$ whose image is $E^{**}$. It is shown in \cite[Theorem 3.9]{EdRutt2001} shows that $F\subseteq \mathcal{B}_{E^{**}}$ is a proper weak$^*$-closed face if and only if $\mathfrak{F}:=P^{-1}(F)\cap \mathcal{B}_{X^{**}}$ is a proper weak$^*$-closed face of $\mathcal{B}_{X^{**}}$. Since $\mathfrak{F}$ is $\tau^{\sharp\sharp}$-symmetric and $F = \mathfrak{F}^{\tau^{\sharp\sharp}}\cap E$, it is not hard to check that $\overline{\mathfrak{F}\cap X}^{w^*}  = \mathfrak{F}$ if and only if $\overline{{F}\cap E}^{w^*}  = {F}$. Therefore the desired equivalence is a consequence of Theorem \ref{t characterization relatively open weak* closed faces} and Corollary \ref{c tripotent second dual is compact iff it is compact in the complexification}.
\end{proof}

It remains to determine the weak$^*$-closed faces in the closed unit ball of the dual space of a  real JB$^*$-triple.

\begin{theorem}\label{t weak* closed face in the dual space real} Let $\tau$ be a conjugation on a JB$^*$-triple $X$, and let $E= X^{\tau}$. Then for each weak$^*$-closed proper face $F$ of $\mathcal{B}_{E^*}$ there exists a unique compact tripotent $u\in E^{**}$ satisfying $F =\{u\}_{\prime,E^*}$. Furthermore, the mapping $$ u \mapsto \{u\}_{\prime,E^*}$$ is an order isomorphism from $\tilde{\mathcal U}_c (E^{**})$ onto $\mathcal{F}_{w^*}(\mathcal{B}_{E})$.
\end{theorem}

\begin{proof} As before, we set $P= \frac{1}{2}(Id_X+\tau)$ and $Q= \frac{1}{2}(Id_X+\tau^{\sharp})$. Then $P$ and $Q$ are contractive real linear projections on $X$ and $X^{*}$ whose images are $E$ and $E^{*}$, respectively, and $Q$ is weak$^*$-continuous. The set $F$ is a weak$^*$-closed proper face of $\mathcal{B}_{E^*}$ if and only if the set $\mathfrak{F}:=Q^{-1}(F)\cap \mathcal{B}_{X^*}$ is a weak$^*$-closed proper face of $\mathcal{B}_{X^*}$. By \cite[Theorem 2]{FerPe10} there exists a (unique) compact tripotent $u\in X^{**}$ satisfying $\mathfrak{F} = \{u\}_{\prime,X^*}$. Clearly, $\mathfrak{F}$ is $\tau^{\sharp}$-symmetric and $F = \mathfrak{F}^{\tau^{\sharp}}  = \mathfrak{F} \cap E^*$. We have commented in previous pages that $\tau$ and $\tau^{\sharp\sharp}$ are triple automorphisms on $X$ and $X^{**}$, respectively. Then, we can easily check that $$\{u\}_{\prime,X^*} = \mathfrak{F} =  \tau^{\sharp} \left( \mathfrak{F}\right) =  \tau^{\sharp} \left( \{u\}_{\prime,X^*} \right) =  \{\tau^{\sharp\sharp} \left( u\right) \}_{\prime,X^*},$$ witnessing that $\tau^{\sharp\sharp} \left( u\right) = u$. Corollary \ref{c tripotent second dual is compact iff it is compact in the complexification} proves that $u$ is a compact tripotent in $E^{**}$. Finally, $F = \mathfrak{F} \cap E^* = \{u\}_{\prime,E^*}.$
\end{proof}

\section{New spaces satisfying a Krein--Milman type theorem}\label{sec: 3}

A convex subset $\mathcal{K}$ of a normed space $X$ is called a \emph{convex body} if it has non-empty interior in $X$. The Mazur--Ulam theorem establishes that every surjective isometry between normed spaces is always affine. In \cite{Mank1972} P. Mankiewicz extended this result by showing that any surjective isometry defined between convex bodies in two arbitrary normed spaces is the restriction of a unique affine isometry between the whole spaces. Mankiewicz's theorem has become a fundamental tool for researchers working on positive solutions to Tingley's problem or on new Banach spaces satisfying the Mazur--Ulam property.\smallskip

In relation with these questions, M. Mori and N. Ozawa have recently contributed by introducing a new point of view 
(see \cite{MoriOza2018}). Following the just quoted authors, we shall say that a convex subset $\mathcal{K}$ of a normed space $X$ satisfies the \emph{strong Mankiewicz property} if every surjective isometry $\Delta$ from $\mathcal{K}$ onto an arbitrary convex subset $L$ in a normed space $Y$ is affine. Every convex subset of a strictly convex normed space satisfies the strong Mankiewicz property because it is uniquely geodesic (see \cite[Lemma 6.1]{BadFurGeMo2007}), and there exist examples of convex subsets of $L^1[0, 1]$ which do not satisfy this property (see \cite[Example 5]{MoriOza2018}). 
In \cite[Theorem 2]{MoriOza2018} Mori and Ozawa establish the following variant of Mankiewicz's theorem.

\begin{theorem}\label{t Mori-Ozawa strong Mankiewicz}\cite[Theorem 2]{MoriOza2018} Let $X$ be a Banach space such that the closed convex hull of the extreme points, $\partial_e (\mathcal{B}_X),$ of the closed unit ball, $\mathcal{B}_X$, of $X$ has non-empty interior in $X$. Then, every convex body $\mathcal{K}\subset X$ has the strong Mankiewicz property. Furthermore, suppose $L$ is a convex subset of a normed space $Y$, and $\Delta : \mathcal{B}_{X}\to L$ is a surjective isometry. Then $\Delta$ can be uniquely extended to an affine isometry from $X$ onto a norm-closed subspace of $Y$.$\hfill\Box$
\end{theorem}

By combining the previous result with the Russo--Dye theorem, Mori and Ozawa proved that every convex body in
unital C$^*$-algebra or in a real von Neumann algebra satisfies the strong Mankiewicz property (see \cite[Corollary 3]{MoriOza2018}). A deeper application of the facial structure of unital C$^*$-algebras leads Mori and Ozawa to a significant achievement in the study of the Mazur--Ulam property.

\begin{theorem}\label{t MO Mazur--Ulam unital Cstar algebras}{\rm\cite[Theorem 1]{MoriOza2018}} Every unital complex C$^*$-algebra {\rm(}as a real Banach space{\rm)} and every real von Neumann algebra has the Mazur--Ulam property.
\end{theorem}

It is worth mentioning that concerning the strong Mankiewicz and the Mazur--Ulam properties, a version of the Mori-Ozawa theorem has been recently established in the wider setting of JBW$^*$-triples.

\begin{theorem}\label{t BCFP Mazur--Ulam JBW*-triples}\cite[Corollary 2.2, Theorem 4.14 and Proposition 4.15]{BeCuFerPe2018}
Every convex body in a JBW$^*$-triple satisfies the strong Mankiewicz property. Every JBW$^*$-triple which is not a Cartan factor of rank two satisfies the Mazur--Ulam property.
\end{theorem}

The previous two theorems reveal the noticeable applicability of Theorem \ref{t Mori-Ozawa strong Mankiewicz} in the study of those problems asking for extension of isometries between the spheres of two Banach spaces. This powerful tool is limited to those Banach spaces whose closed unit ball coincides with the closed convex hull of its extreme points. For this reason, we survey some forerunners where the latter property has been studied.\smallskip

W.G. Bade proved that $\hbox{co}(\partial_{e}\mathcal{B}_{C(K,\RR)})$ is dense in the closed unit ball of the space ${C(K,\RR)}$, of all real-valued continuous functions on a compact Hausdorff space $K$, if and only if $K$ is totally disconnected (see \cite{Bade57}). 
The complex case was considered by R.R. Phelps in \cite{Phelps65}, where he showed that the closed unit ball of the commutative unital C$^*$-algebra $C(K)$ coincides with the closed convex hull of its extreme points. Since the extreme points of the closed unit ball of $C(K)$ are precisely the unitary elements in $C(K)$, Phelps provided in fact a particular case of the celebrated Russo--Dye theorem (\cite{RuDye}), which states that the closed unit ball of any unital C$^*$-algebra agrees with the closed convex hull of its unitary elements.\smallskip

When the complex field is replaced with a general Banach space $X$ with dim$(X)\geq 3$, the notion of unitary element does not make any sense in the space $C(K,X),$ of all $X$-valued continuous functions on $K$. In the setting of $C(K,X)$ spaces the problem of determining whether its closed unit ball coincides with the closed convex hull of its extreme points was explored by authors like J. Cantwell \cite{Cantwell68}, N.T. Peck \cite{Peck1968}, J.F. Mena-Jurado, J.C. Navarro-Pascual and V.I. Bogachev \cite{MeNav1995,BoMeNav}. 
Since the notion of unitary is no longer applicable, these results are called Krein--Milman type theorems.\smallskip

All the comments above provide sufficient motivation for identifying new examples of Banach spaces satisfying a Krein--Milman type theorem. Some of them can be obtained by certain ``hyperplanes'' associated with multiplicative functionals on unital C$^*$-algebras. Let $A$ be a unital C$^*$-algebra and suppose $\varphi: A\to \CC$ is a homomorphism. 
We observe first that $\varphi$ is automatically continuous (cf. \cite[\S 16, Proposition 3]{BonsDun73}). We can therefore apply the Gleason-Kahane-\.{Z}elazko theorem \cite[Theorem 2]{Zel68} to deduce that $\varphi$ is in fact a $^*$-homomorphism, that is,  $\varphi(a^*)=\varphi(a)^*$, for every element $a$ in $A$.
Consequently, $\varphi$ is a triple homomorphism for the triple product when $A$ and $\mathbb{C}$ both are equipped with the tripe product defined in \eqref{eq product operators}. However, given $\lambda\in\TT$, the non-zero functional $\psi=\lambda\varphi: A \to \CC$ is a triple homomorphism which is not multiplicative.  
\smallskip

It is worth noting that every triple homomorphism $\psi: A \to \CC$ can be expressed as a product of an element $\lambda\in \mathbb{T}$ and a $^*$-homomorphism $\varphi:A\to\CC$. We note that every triple homomorphism $\psi$ from a JB$^*$-triple $E$ into $\CC$ is automatically continuous (cf. \cite[Lemma 1.6]{Ka83}). Suppose $\psi\neq 0$. Since for every $a\in A$ we have $\psi(a)= \psi \{a,1,1 \}= \{\psi(a),\psi(1),\psi(1) \}=\psi(a)\overline{\psi(1)} \psi(1)$, it follows that $\psi(1)\in\TT$ because $\psi\neq 0$. It is standard to check that the mapping $\varphi=\overline{\psi(1)}\psi$ is a Jordan $^*$-homomorphism from $A$ onto $\CC$. We can therefore apply \cite[proof of Theorem 1]{Zel68} to deduce that $\varphi$ is a $^*$-homomorphism, and $\psi=\psi(1)\varphi$.\smallskip

Let $A$ be a unital C$^*$-algebra, let $\varphi:A\to \CC$ be a (continuous) multiplicative functional, and let $A^\varphi_\RR:= \varphi^{-1}(\RR)=\{ a\in A : \varphi(a)\in \RR \}$. Clearly $A^\varphi_\RR$ is a real C$^*$-subalgebra of $A$.  M. Mori and N. Ozawa prove in \cite[Lemma 19]{MoriOza2018} that $\mathcal{B}_{A^\varphi_\RR}$ coincides with the closed convex hull of the unitary elements in $A^\varphi_\RR$. The next statement somehow extends this conclusion to the triple setting. 
The result also shows a new class of real JB$^*$-triples satisfying a Krein--Milman type theorem.

\begin{proposition}\label{p triple homomorphism and real JBstar-triple strong Mankiewicz property}
Let $A$ be a unital C$^*$-algebra and let $\psi: A\to \CC$ be a {\rm(}continuous{\rm)} non-zero triple homomorphism. Then the closed unit ball of the real JB$^*$-triple $A^\psi_\RR:= \psi^{-1}(\RR)$ coincides with the closed convex hull of unitary tripotents in  $A^\psi_\RR$. Consequently, $\mathcal{B}_{A^\psi_\RR}$ and every convex body $\mathcal{K}\subset A^\psi_\RR$ satisfy the strong Mankiewiecz property.
\end{proposition}

\begin{proof}
The observations made above guarantee the existence of a non-zero and (continuous) multiplicative functional $\varphi: A \to \CC$ and an element $\lambda$ in $\TT$ such that $\psi=\lambda\varphi$. If we write $A^\varphi_\RR=\{ b\in A : \varphi(b)\in \RR\}$, it is clear that $A^\psi_\RR=\psi^{-1}(\RR)=(\lambda\varphi)^{-1}(\RR)=\{ a\in A : \varphi(a)\in \overline{\lambda}\RR\} =\overline{\lambda}A^\varphi_\RR$.
Therefore, $\mathcal{B}_{A^\psi_\RR}=\mathcal{B}_{\overline{\lambda}A^\varphi_\RR} =\overline{\lambda}\mathcal{B}_{A^\varphi_\RR}$.\smallskip
	
Let us pick now $a\in \mathcal{B}_{A^\psi_\RR}$ and $\varepsilon>0$. We have shown that there exists $b\in \mathcal{B}_{A^\varphi_\RR}$ such that $a=\overline{\lambda}b$. It is shown in the proof of \cite[Lemma 19]{MoriOza2018} that there exist unitary elements $u_1,\ldots,u_n$ in the real C$^*$-algebra $A^\varphi_\RR$ and $\alpha_1,\dots,\alpha_n$ in $[0,1]$ with $\sum\limits_{j=1}^{n}\alpha_j=1$ satisfying $\left\|b-\sum\limits_{j=1}^{n}\alpha_ju_j\right\|<\varepsilon$. Therefore, $\left\|a-\sum\limits_{j=1}^{n}\alpha_j\overline{\lambda}u_j\right\|=\left\| \overline{\lambda}b-\sum\limits_{j=1}^{n}\alpha_j\overline{\lambda}u_j \right\|=\left\|b-\sum\limits_{j=1}^{n}\alpha_ju_j\right\|<\varepsilon$.
	Finally, we observe that $\overline{\lambda}u_1,\ldots,\overline{\lambda}u_n$ are unitary tripotents in $A^\psi_\RR$. 
The final conclusion follows from Theorem \ref{t Mori-Ozawa strong Mankiewicz} and  \cite[Lemma 4]{MoriOza2018}.
\end{proof}

A Krein-Milan type theorem for the space $C(K,\mathcal{H})$ is essentially known in the literature.

\begin{proposition}\label{p conditions on C(K,H) to satisfy the Strong Mankiewicz property}\cite{Phelps65,Cantwell68,Peck1968}
Let $K$ be a compact Hausdorff space and let $\mathcal{H}$ be a real Hilbert space with dim$(\mathcal{H})\geq 2$. Then the closed unit ball of $C(K,\mathcal{H})$ coincides with the closed convex hull of its extreme points. Consequently, every convex body in $C(K,\mathcal{H})$ satisfies the strong Mankiewicz property.
\end{proposition}

\begin{proof}
If dim$(\mathcal{H})=n\in \NN$, we can identify the Hilbert space $\mathcal{H}$ with $\ell_{2}^n(\mathbb{R})$. If $n=2$, R. Phelps proved in \cite[Theorem 1]{Phelps65} that the convex hull of the extreme points of $B_{C(K)}$ is always dense in the closed unit ball. If $n>2$, the same conclusion holds by \cite[Theorem I and Remark]{Cantwell68}. On the other hand, if dim$(\mathcal{H})=\infty$, then the closed unit ball of $C(K,\mathcal{H})$ coincides with the convex hull of its extreme points by \cite[Theorem 5]{Peck1968}. Therefore, in both cases $C(K,\mathcal{H})$ satisfies a Krein--Milman type theorem and the thesis of our proposition derives from Theorem \ref{t Mori-Ozawa strong Mankiewicz}.
\end{proof}

The following technical lemma is required for later purposes.

\begin{lemma}\label{l approximation by non-vanishing functions}
Let $K$ be a compact Hausdorff space and let $\mathcal{H}$ be a real Hilbert space with dim$(\mathcal{H})=n\geq 2$. Suppose $t_0\in K$ and $x_0\in S(\mathcal{H})$. If $a\in \mathcal{B}_{C(K,\mathcal{H})}$ is such that $a(t_0)\in\RR x_0$ and $\varepsilon>0$ is small enough. Then the following statements hold:
\begin{enumerate}[$(a)$]\item If $\mathcal{H}$ is infinite dimensional, then there exists a non-vanishing function $b$ in $\mathcal{B}_{C(K,\mathcal{H})}$ such that $b(t_0)\in\RR x_0$ and $\|a-b \|<\varepsilon$. If $a(t_0)\neq 0$, we can also assume that $b(t_0) = a(t_0)$;
\item If $\mathcal{H}$ is finite dimensional, then there exist non-vanishing continuous functions $b_1,\ldots,b_k$ in $\mathcal{B}_{C(K,\mathcal{H})}$ such that $b_j(t_0)\in\RR x_0$, for every $j\in \{1,\ldots,k\}$, and $\left\|a-\frac{1}{k}\sum_{j=1}^{k} b_j\right\|\leq \varepsilon.$ If $a(t_0)\neq 0$, we can also assume that $b_j(t_0) = a(t_0)$ for all $j\in \{1,\ldots,k\}.$

    Furthermore, for each $j$ in $\{0,\ldots,k\}$ there exit $v_j\in C(K,\mathcal{H})$ satisfying $\|v_j (t) \|= 1,$ and $(b_j(t)|v_j(t))= 0,$ for all $t\in K,$ and thus $u_j = b_j + (1-\|b_j(\cdot)\|^2)^{\frac12} v_j$, $w_j = b_j - (1-\|b_j(\cdot)\|^2)^{\frac12} v_j$ both lie in $\partial_e(\mathcal{B}_{C(K,\mathcal{H})})$ and $b_j = \frac{1}{2}(u_j +w_j)$.
\end{enumerate}

\end{lemma}

\begin{proof}
Take $a\in \mathcal{B}_{C(K,\mathcal{H})}$ such that $a(t_0)=\lambda x_0$, with $\lambda\in \RR$, and $\varepsilon>0$. We shall split the proof into two cases.\smallskip

\noindent\emph{Case 1:} Suppose $\mathcal{H}$ is infinite dimensional.\smallskip
		
If $\lambda\in \RR \setminus \{0\}$, then clearly $1\geq \|a\|\geq |\lambda|>0$. By \cite[Corollary after Proposition 2 ]{Peck1968} applied to $|\lambda|/2 > \varepsilon/2>0$, there exists $b\in \mathcal{B}_{C(K,\mathcal{H})}$ which is a non-vanishing function (i.e. $\|b(t)\|\geq m>0$ for every $t\in K$, and some $m\in\RR^+$) such that for each $t\in K$, $\|b(t)\|<\varepsilon/2 \ $ if $\ \|a(t)\|<\varepsilon/2$, and $b(t)=a(t) \ $ if $\ \|a(t)\|\geq\varepsilon/2$. It is not hard to check that $\|a-b\|<\varepsilon$, and $b(t_0)=a(t_0)=\lambda x_0$ because $\|a(t_0)\| = |\lambda|>\varepsilon/2$.\smallskip
		
On the other hand, if $\lambda=0$, that is, if $a(t_0)=0$, let us consider the open set $\mathcal{U}_\varepsilon=\{t\in K : \|a(t)\|<\varepsilon/2\}$. By Urysohn's lemma there exists a continuous function $f:K\to \RR$ such that $0\leq f \leq 1$, $f(t_0)=1$ and $f|_{K\setminus \mathcal{U}_\varepsilon}\equiv 0$. Define $\widetilde{a} = a + (\varepsilon/2) x_0 \otimes f\in C(K,\mathcal{H})$, which lies in the closed unit ball for $\varepsilon$ small enough ($\varepsilon\leq 1$). Note that $\|a-\widetilde{a}\|\leq \varepsilon/2$. Since $\widetilde{a}(t_0)=(\varepsilon/2) x_0$ and $\varepsilon/2\neq 0$, we have shown before that there exists a non-vanishing function $b\in \mathcal{B}_{C(K,\mathcal{H})}$ such that for each $t\in K$, $\|b(t)\|<\varepsilon/4 \ $ if $\ \|\widetilde{a}(t)\|<\varepsilon/4$, and $b(t)=\widetilde{a}(t) \ $ if $\ \|\widetilde{a}(t)\|\geq\varepsilon/4$ (cf. \cite[Corollary after Proposition 2 ]{Peck1968}). Therefore $b(t_0)=\widetilde{a}(t_0)=(\varepsilon/2)x_0$. It is also clear that $\|\widetilde{a}-b\|<\varepsilon/2$, and thus $\|a-b\|\leq \|a-\widetilde{a}\|+ \|\widetilde{a}-b\|<\varepsilon/2+\varepsilon/2=\varepsilon$ as desired.\smallskip
	
\noindent\emph{Case 2} Suppose dim$(\mathcal{H})=n\geq 2$.\smallskip
	
As before, let us distinguish the cases $\lambda=0$ and $\lambda\neq 0$. Let us first assume that $|\lambda |\geq 2\varepsilon>0$ with $\varepsilon$ small enough. Following the arguments due to R.C. Sine and N.T. Peck (see \cite[proof of Theorem 1]{Peck1968}), for $\alpha, \beta>0$ and $z_0\in S(\mathcal{H})$, we shall consider $B(z_0,\alpha)=\{ z\in S(\mathcal{H}): \| z-z_0 \|<\alpha \}$ and the wedge $W(z_0,\alpha,\beta):= \hbox{co}(B(z_0,\alpha)\cup \{-\beta z_0\})$.\smallskip
	
For every $\varepsilon>0$, there exists $k\in \NN$ such that $\frac{1}{k}<\frac{\varepsilon}{2}$. Find $z_1,\cdots, z_k\in S(\mathcal{H})$, $\alpha_1,\cdots, \alpha_k\in \RR^+$ and $\beta_1,\cdots, \beta_k\in \RR^+$, satisfying:
\begin{enumerate}[$\bullet$]\item The sets $\{ W(z_j,\alpha_j,\beta_j): j =1,\ldots, k\}$ are pointwise disjoint outside the closed ball in $\mathcal{H}$ centered in zero with radius $\varepsilon/2$;
\item $W(x_j,\alpha_j,\beta_j)\cap \RR x_0 \subseteq [-\varepsilon x_0,\varepsilon x_0]$, for every $j=1,\ldots,k$.
\end{enumerate}

Let us now define $\varphi_j:\mathcal{B}_{\mathcal{H}}\to \mathcal{B}_{\mathcal{H}}\setminus \mathring{W}(z_j,\alpha_j,\beta_j)$ given by $\varphi_j(z)=z$ if $z\notin W(z_j,\alpha_j, \beta_j)$, and for $z\in W(z_j,\alpha_j,\beta_j)$, $\varphi_j(z)$ is obtained by projecting $z$ parallel to $-z_j$ until it hits the boundary of $W(z_j,\alpha_j,\beta_j)$. The number $\beta_j$ can be chosen such that $\| \varphi_j(z) \|\leq \varepsilon/2$, for every $\|z\|\leq \varepsilon/2$. We claim that \begin{equation}\label{eq 1 200219} \left\|z-\frac{1}{k}\sum_{j=1}^{k}\varphi_j(z)\right\|\leq \varepsilon, \hbox{ for every $z\in \mathcal{B}_{\mathcal{H}}$.}
 \end{equation}Namely, if we take $z\in \mathcal{B}_{\mathcal{H}}$ with $\|z\|\leq \varepsilon/2$, then $\left\|z-\frac{1}{k}\sum_{j=1}^{k}\varphi_j(z)\right\|\leq \|z\| + \frac{1}{k}\sum_{j=1}^{k}\|\varphi_j(z)\|\leq \varepsilon$. On the other hand, if we pick $z\in \mathcal{B}_{\mathcal{H}}$ with $\|z\|> \varepsilon/2$, then $z$ lies in at most one $W(z_{j_0}, \alpha_{j_0}, \beta_{j_0})$, and that implies $\left\|z-\frac{1}{k}\sum_{j=1}^{k}\varphi_j(z)\right\|=\left\| \frac{z}{k}-\frac{\varphi_{j_0}(z)}{k} \right\|\leq \frac{2}{k}<\varepsilon$, as we were expecting.\smallskip
	
Set $b_j :=\varphi_j\circ a\in \mathcal{B}_{C(K,\mathcal{H})}$. Obviously, $b_j$ is a non-vanishing function. It follows from \eqref{eq 1 200219} that $\displaystyle \left\|a-\frac{1}{k}\sum_{j=1}^{k}b_j\right\|= \left\|a-\frac{1}{k}\sum_{j=1}^{k}\varphi_j\circ a\right\|\leq \varepsilon.$ Furthermore, since $\|a(t_0)\|=|\lambda|\geq 2\varepsilon>\varepsilon$ and $W(x_j,\alpha_j,\beta_j)\cap \RR x_0 \subseteq [-\varepsilon x_0,\varepsilon x_0]$, it follows that $a(t_0)=\lambda x_0\notin W(z_j,\alpha_j,\beta_j)$, for any $j\in\{1,\cdots,k\}$, and thus $b_j(t_0)=\varphi_j\circ a (t_0)=a(t_0)=\lambda x_0\in \RR x_0$, for every $j\in\{1,\cdots,k\}$.\smallskip
	
If we assume $a(t_0)=0$, we can argue as in the infinite dimensional case, and thus, for $0<\varepsilon<1$ we define $\widetilde{a}\in \mathcal{B}_{C(K,\mathcal{H})}$, with $\widetilde{a}(t_0)=(\varepsilon/2) x_0$ and such that $\|a-\widetilde{a}\|\leq \varepsilon/2$. Now we can apply the conclusions above which guarantee the existence of $k\in \NN$,  non-vanishing functions $b_1,\cdots, b_k\in \mathcal{B}_{C(K,\mathcal{H})}$ and such that $b_j(t_0)=(\varepsilon/2) x_0$, for every $j\in\{1,\cdots,k\}$. The desired conclusion follows from the inequality $\displaystyle \left\| a-\frac{1}{k}\sum_{j=1}^{k}b_j \right\|\leq \left\| a-\widetilde{a} \right\| + \left\|\widetilde{a}- \frac{1}{k}\sum_{j=1}^{k}b_j \right\|<\varepsilon$.\smallskip

The rest of the argument is essentially in \cite[Proof of Theorem 1]{Peck1968}. It is shown in the just quoted paper that, for each $j\in\{1,\cdots,k\}$ there exists a continuous field $\vartheta_j : S(\mathcal{H})\backslash B(z_j,\alpha_j)\to S(\mathcal{H})$ (i.e. a continuous mapping satisfying $(\vartheta_j (z) | z) = 0$ for all $z\in S(\mathcal{H})\backslash B(z_j,\alpha_j)$). Taking $v_j :=\vartheta_j(\frac{b_j (\cdot)}{\|b_j(\cdot)\|})$ we get the desired statement.
\end{proof}

Let $K$ be a compact Hausdorff space, and let $H$ be a real or complex  Hilbert space. For each $t_0\in K$ and each $x_0\in S(H)$ we set $$A(t_0,x_0):=\{ a\in S(C(K,H)) : a(t_0) = x_0\}.$$ It is not hard to check that $A(t_0,x_0)$ is a maximal norm-closed proper face of $\mathcal{B}_{C(K,H)}$ and a maximal convex subset of $S(C(K,H))$. Actually, every maximal convex subset of the unit sphere of $C(K,H)$ is of this form.\smallskip

Our next corollary is one of the main technical tools required for our main result.

\begin{corollary}\label{c convex combinations of unitaries in maximal faces}
Let $K$ be a compact Hausdorff space and let $H$ be a complex Hilbert space. 
Suppose $t_0\in K$ and $x_0\in S(H)$. Then every element in $A(t_0,x_0)$ can be approximated in norm by a finite convex combination of elements in $A(t_0,x_0)\cap \partial_{e}(\mathcal{B}_{C(K,H)})$.
\end{corollary}

\begin{proof} Take $a\in A(t_0,x_0)$. By Lemma \ref{l approximation by non-vanishing functions} can be approximated in norm by a finite convex combination of non-vanishing functions in $A(t_0,x_0)\cap \mathcal{B}_{C(K,H)}$. Let $b$ be a non-vanishing functions in $A(t_0,x_0)\cap \mathcal{B}_{C(K,H)}$. The element $u(\cdot) = \frac{b(\cdot)}{\|b(\cdot)\|}$ lies in $A(t_0,x_0)$ and is a maximal tripotent in $C(K,H)$ (i.e., an element in $A(t_0,x_0)\cap \partial_{e}(\mathcal{B}_{C(K,H)})$).\smallskip

To simplify the notation let us write $E$ for $C(K,H)$. Clearly, $b$ is a hermitian element in the JB$^*$-algebra $E_2(u)$. Let $A$ denote the JB$^*$-subalgebra of $E_2(u)$ generated by $b$ and $u$. It is known that $A$ is isometrically isomorphic to a commutative unital C$^*$-algebra (cf. \cite[Theorem 3.2.4]{HOS}). The intersection $F= A(t_0,x_0)\cap A\subseteq S(A)$ is a maximal norm-closed face of $\mathcal{B}_{A}$. Lemma 18 in \cite{MoriOza2018} guarantees that $b\in F$ can be approximated in norm by a finite convex combination of elements in $F\cap \partial_{e}(\mathcal{B}_{A})$. Every element in $\partial_{e}(\mathcal{B}_{A})$ is a unitary element in $A$, and hence a unitary element in $E_2(u)$. We further know from Lemma 4 in \cite{Sidd2007} that every unitary element in $E_2(u)$ is an extreme point of $\mathcal{B}_{E}$. We can therefore conclude that $F\cap \partial_{e}(\mathcal{B}_{A}) \subseteq F\cap \partial_{e}(\mathcal{B}_{E})\subseteq A(t_0,x_0)\cap \partial_{e}(\mathcal{B}_{E}),$ which finishes the proof.
\end{proof}

The case of real Hilbert spaces is treated in the next result.

\begin{corollary}\label{c convex combinations of unitaries in maximal faces real}
Let $K$ be a compact Hausdorff space and let $\mathcal{H}$ be a finite-dimen-sional real Hilbert space with dim$(\mathcal{H})=n\geq 2$.
Suppose $t_0\in K$ and $x_0\in S(\mathcal{H})$. Then every element in $A(t_0,x_0)$ can be approximated in norm by a finite convex combination of elements in $A(t_0,x_0)\cap \partial_{e}(\mathcal{B}_{C(K,\mathcal{H})})$.
\end{corollary}

\begin{proof} Let $a$ be an element in $A(t_0,x_0)$. Since $a(t_0) = x_0\in S(\mathcal{H})$, Lemma \ref{l approximation by non-vanishing functions}$(b)$, for each $\varepsilon>0$ small enough, there exist non-vanishing continuous functions $b_1,\ldots,b_k$ in $\mathcal{B}_{C(K,\mathcal{H})}$ such that $b_j(t_0)= a(t_0)= x_0$, for every $j\in \{1,\ldots,k\}$, and $\left\|a-\frac{1}{k}\sum_{j=1}^{k} b_j\right\|\leq \varepsilon.$ Furthermore, for each $j$ in $\{0,\ldots,k\}$ there exit $v_j\in C(K,\mathcal{H})$ satisfying $\|v_j (t) \|= 1,$ and $(b_j(t)|v_j(t))= 0,$ for all $t\in K,$ and thus $u_j = b_j + (1-\|b_j(\cdot)\|^2)^{\frac12} v_j$, $w_j = b_j - (1-\|b_j(\cdot)\|^2)^{\frac12} v_j$ both lie in $\partial_e(\mathcal{B}_{C(K,\mathcal{H})})$ and $b_j = \frac{1}{2}(u_j +w_j)$. Having in mind that $\|a(t_0)\|= \|x_0\|=1$, we can easily see that $u_j(t_0) = w_j(t_0) = b_j(t_0) = a(t_0) = x_0$, which guarantees that $u_j,w_j\in A(t_0,x_0)\cap \partial_{e}(\mathcal{B}_{C(K,\mathcal{H})})$. Finally, $$\left\| a - \frac{1}{2 k}\sum_{j=1}^{k} (u_j+w_j) \right\| = \left\|a-\frac{1}{k}\sum_{j=1}^{k} b_j \right\|\leq \varepsilon,$$ as desired.
\end{proof}

Let $H$ be a real or complex Hilbert space, and let $K$ be a compact Hausdorff space. 
Let $\mathcal{O}\neq \emptyset$ be an open subset of $K$. We set $p:= \chi_{_\mathcal{O}}$ the characteristic function of $\mathcal{O}$. Note that we cannot, in general, assume that $p\in C(K)$.
\smallskip

Fix $x_0\in S(H)$. Let us define some subsets of $C(K,H)$ whose elements are constant on $A$:

\begin{align}\label{eq subtriples F,B,N}
F&=F_{x_0\otimes p}:= \{ a\in S(C(K,H)) : ap= x_0 \otimes p \},\\
B&=B_p:=\{ a\in C(K,H) : ap=h \otimes p, \text{ for some }  h\in H \}, \nonumber\\
N&=N_p^{x_0}:= \{ a\in C(K,H) : ap= \mu x_0 \otimes p, \text{ for some }  \mu\in \mathbb{K} \},\nonumber
\end{align} where $\mathbb{K}=\mathbb{R}$ if $H$ is a real Hilbert space and $\mathbb{K}=\mathbb{C}$ if $H$ is a complex Hilbert space.


It is not difficult to check that $N$ and $B$ are norm-closed subtriples of $C(K,H)$ with $F\subseteq N \subseteq B \subseteq C(K,H)$.\smallskip

Let us consider the mapping $T:B\to H$ defined by $T(a)=a(t_0)$ for each $a\in B$, where $t_0$ is any element in the open set $\mathcal{O}$. Clearly $T$ is linear. We further know that $T$ is a triple homomorphism. Indeed, if we take $a,b,c\in B$ and write $a(t_0)=x_a$, $b(t_0)=x_b$ and $c(t_0)=x_c$ the constant elements in the Hilbert space associated to each function, we have that
$$T\{a,b,c\}=\{a,b,c\}(t_0)= \frac12 \langle a(t_0) | b(t_0) \rangle c(t_0) + \frac12 \langle c(t_0) | b(t_0) \rangle a(t_0)$$
$$=\frac12 \langle x_a | x_b \rangle x_c + \frac12 \langle x_c | x_b \rangle x_a=\{x_a,x_b,x_c\}=\{T(a),T(b),T(c)\}.$$
The restriction $T|_N:N\to \mathbb{K} x_{0}\subseteq H$ also is a triple homomorphism and $T|_N(a)=a(t_0)=\mu_a x_0 $ for every $a\in N$, where $\mu_a\in\mathbb{K}$.\smallskip

We are now in position to present an extension of \cite[Lemma 19]{MoriOza2018} and Proposition \ref{p triple homomorphism and real JBstar-triple strong Mankiewicz property} to the setting of continuous functions valued in a Hilbert space.

\begin{proposition}\label{p N strong Mankiewicz property}
Let $K$ be a compact Hausdorff space and let $H$ be a complex Hilbert space. Suppose $x_0\in S(H)$ and $\mathcal{O}\neq \emptyset$ is an open subset of $K$. Let us denote $p= \chi_{_\mathcal{O}},$ $N= \{ a\in C(K,H) : ap= \mu x_0 \otimes p, \text{ for some }  \mu\in \mathbb{C} \}$, and $\varphi: N\to \mathbb{C}$ the triple homomorphism defined by $\varphi(a) = \langle a(t_0) | x_0\rangle$ {\rm(}$a\in N${\rm)}, where $t_0$ is any element in $\mathcal{O}$. Then the closed unit ball of $N_\RR ^\varphi:= \varphi^{-1}(\RR)$ coincides with  the closed convex hull of its extreme points. Consequently, $\mathcal{B}_{N_\RR^\varphi}$ satisfies the strong Mankiewicz property.
\end{proposition}

\begin{proof} Let us pick $a\in\mathcal{B}_{N_\RR^\varphi}$ and any $t_0\in \mathcal{O}$. Since $\varphi(a)=\lambda \in \RR$, we can assure that $a(t_0) = \lambda x_0\in \mathbb{R} x_0$. Without loos of generality, we can assume, via Lemma \ref{l approximation by non-vanishing functions}, that $a$ is a non-vanishing function. Define now $u\in C(K,H)$ given by $u(t):=\frac{a(t)}{\|a(t)\|}$, for every $t\in K$. Observe that $u\in\partial_{e}(\mathcal{B}_{C(K,H)})$. We further know that $u$ lies in $N^\varphi_\RR$ because $u p = \frac{a(t_0)}{|a(t_0)|} x_0 \otimes p$ and $\varphi (u) = \frac{a(t_0)}{|a(t_0)|}.$\smallskip
	
To simplify the notation we set $E= C(K,H)$. Since $a\in E^1 (u)$ (actually $a\in N^1 (u)$), the JB$^*$-subtriple of $E$ generated by $a$ and $u$ is JB$^*$-triple isomorphic (and hence isometric) to a commutative unital C$^*$-algebra. Let $A$ denote the JB$^*$-subtriple generated by $a$ and $u$. Since $a,u\in N$, it follows that $A\subseteq N $, and hence the restriction $\varphi|_A:A\to \CC$ is a non-zero triple homomorphism. By applying Proposition \ref{p triple homomorphism and real JBstar-triple strong Mankiewicz property} to the real JB$^*$-triple $A^{\varphi|_A}_\RR$ we conclude that $a\in A^{\varphi|_A}_\RR$ can be approximated in norm by convex combinations of unitary tripotents in ${{A^{\varphi|_A}_\RR}}$.\smallskip

Finally, every unitary element in ${{A^{\varphi|_A}_\RR}}$ is a unitary element in the JB$^*$-algebra $E_2(u)$. We observe now that, since $u$ is an extreme point of $\mathcal{B}_{E}$, every unitary element in $E_2(u)$ is an extreme point of $\mathcal{B}_{E}$ (cf. \cite[Lemma 4]{Sidd2007}), and thus every unitary element in ${{A^{\varphi|_A}_\RR}}$ belongs to $\partial_{e}(\mathcal{B}_{N^{\varphi}_\RR})$, because $A^{\varphi|_M}_\RR\subseteq N^\varphi_\RR$. Theorem \ref{t Mori-Ozawa strong Mankiewicz} gives the final statement.
\end{proof}

We shall next establish a real version of Proposition \ref{p N strong Mankiewicz property}. For reasons which will be better understood at the end of the next section, we shall restrict our interest to the finite dimensional case.

\begin{proposition}\label{p N strong Mankiewicz property rel Hilbert spaces}
Let $K$ be a compact Hausdorff space and let $\mathcal{H}$ be a finite-dimensional real Hilbert space with dim$(\mathcal{H})\geq 2$. Suppose $x_0\in S(\mathcal{H})$ and $\mathcal{O}\neq \emptyset$ is an open subset of $K$. Let us denote $p= \chi_{_\mathcal{O}},$ and $$N= \{ a\in C(K,\mathcal{H}) : a p= \mu x_0 \otimes p, \text{ for some }  \mu\in \mathbb{R} \}.$$ Then the closed unit ball of the real JB$^*$-triple $N$ coincides with the closed convex hull of its extreme points. Consequently, $\mathcal{B}_{N}$ satisfies the strong Mankiewicz property.
\end{proposition}

\begin{proof}
Every function $a\in N$ is constant on the compact subset $\overline{\mathcal{O}}$. By replacing $K$ with the compact quotient space $\tilde{K} = K/ \overline{\mathcal{O}}$, we can assume without loss of generality, that $\overline{\mathcal{O}}$ is a single point $t_0$ in $K$ and $N$ is the real JB$^*$-subtriple of $C(K,\mathcal{H})$ of all functions $a\in C(K,\mathcal{H})$ such that $a(t_0)\in \mathbb{R} x_0$.\smallskip

Let $a\in \mathcal{B}_{N}$. If $a(t_0) =0,$ arguing as in the proof of Lemma \ref{l approximation by non-vanishing functions}$(b)$, for each $0<\varepsilon<1$, we can find $\widetilde{a}\in \mathcal{B}_{C(K,\mathcal{H})}$, with $\widetilde{a}(t_0)=(\varepsilon/2) x_0$ and such that $\|a-\widetilde{a}\|\leq \varepsilon/2$. Clearly, $\tilde{a}\in \mathcal{B}_{N}$ and does not vanish on $t_0$. By  Lemma \ref{l approximation by non-vanishing functions}$(b)$ applied to $\tilde{a}$, there exist non-vanishing continuous functions $b_1,\ldots,b_k$ in $\mathcal{B}_{C(K,\mathcal{H})}$ such that $b_j(t_0)= \tilde{a}(t_0)=(\varepsilon/2) x_0$, for every $j\in \{1,\ldots,k\}$, and $\displaystyle \left\|\tilde{a}-\frac{1}{k}\sum_{j=1}^{k} b_j\right\|\leq \varepsilon.$
Furthermore, for each $j$ in $\{0,\ldots,k\}$ there exit $v_j\in C(K,\mathcal{H})$ satisfying $\|v_j (t) \|= 1,$ and $(b_j(t)|v_j(t))= 0,$ for all $t\in K.$ 
\smallskip

To simplify the notation we write $E=C(K,\mathcal{H})$. Let us fix a non-vanishing function $b\in \mathcal{B}_{E}$ with $b (t_0) \in \mathbb{R} x_0$ and $v\in E$ satisfying $\|v (t) \|= 1,$ and $(b(t)|v(t))= 0,$ for all $t\in K.$ We set $u(\cdot) :=\frac{b(\cdot)}{\|b(\cdot)\|}\in E.$ It is not hard to check that $u$ and $v$ are tripotents in $E$ with $E_2(u) = E^{1} (u) = C(K,\mathbb{R}) u$, $E_2(v) = E^{1} (v) = C(K,\mathbb{R}) v$, $\{u,u,v\} = \frac12 v$, and $\{v,v,u\} = \frac12 u$. \smallskip

Let us consider the real JB$^*$-subtriple $F= E_2(u) \oplus E_2(v) = C(K,\mathbb{R}) u\oplus C(K,\mathbb{R}) v$. Clearly, $N\cap F$ is a non-trivial real JB$^*$-subtriple of $F$ containing $b$. The mapping $\Psi: F\to C(K)$, $\Psi (f u + g v) =f + i g$, is a surjective isometric triple isomorphism between real JB$^*$-triples which maps $N\cap F$ to $C(K)^{t_0}_{\mathbb{R}} = \{h\in C(K) : h(t_0) \in \mathbb{R}\}$. It follows from \cite[Lemma 19]{MoriOza2018} that every element in the closed unit ball of $C(K)^{t_0}_{\mathbb{R}}$ (in particular $\Psi (b)$) can be approximated in norm by convex combinations of unitary tripotents in $C(K)^{t_0}_{\mathbb{R}}$. Finally, if $\upsilon$ is a unitary element in $C(K)^{t_0}_{\mathbb{R}}$, then $w=\Re\hbox{e} (\upsilon) u + \Im\hbox{m} (\upsilon) v\in N\cap F$ with $w(t_0)= \Re\hbox{e} (\upsilon) (t_0) u (t_0) = \upsilon(t_0) u (t_0)\in \mathbb{R} x_0$ and $\|w(t) \|^2_{\mathcal{H}} = | \Re\hbox{e} (\upsilon) (t)|^2 + |\Im\hbox{m} (\upsilon) (t)|^2 = 1,$ for all $t\in K,$ witnessing that $w\in \partial_{e} (\mathcal{B}_{E})$.
\end{proof}

\section{$C(K,H)$ satisfies the Mazur--Ulam property}

Throughout this section $K$ and $H$ will denote a compact Hausdorff space and a complex Hilbert space with dim$(H)\geq 2$, respectively.\smallskip

%
%

Given an element $y_0$ in a Banach space $Y$, we write $\tau_{y_0}$ for the translation by the element $y_0$, that is, $\tau_{y_0} (y ) = y+ y_0$, for all $y\in Y$.\smallskip

Our first lemma is essentially contained in 
\cite{Liu2007} 
and \cite[Lemma 2.1]{CuePer}, and its proof can be easily deduced from the arguments in the just quoted references.

\begin{lemma}\label{l existence of support functionals for the image of a face} Let $\Delta : S(C(K,H))\to S(Y)$ be a surjective isometry, where $Y$ is a real Banach space. Then for each $t_0\in K$ and each $x_0\in S(H)$ the set $$\hbox{supp}_\Delta(t_0,x_0) := \{\psi\in Y^* : \|\psi\|=1,\hbox{ and } \psi^{-1} (\{1\})\cap \mathcal{B}_{Y} = \Delta(A(t_0,x_0)) \}$$ is a non-empty weak$^*$-closed face of $\mathcal{B}_{Y^*}$. $\hfill\Box$
\end{lemma}

In the hypothesis of the previous lemma, it is known that each $A(t_0,x_0)$ is an intersection face in the sense employed in \cite{MoriOza2018}. 
Therefore, Lemma 8 in \cite{MoriOza2018} assures that $\Delta (-A(t_0,x_0)) = - \Delta (A(t_0,x_0)),$ and consequently,
\begin{equation}\label{eq post lemma support} \psi \Delta (a) = -1, \hbox{ for all } a\in -A(t_0,x_0) \hbox{,  and all } \psi\in \hbox{supp}_\Delta(t_0,x_0).
\end{equation}

%

The following technical lemma might be known, although an explicit reference is out from our knowledge. We include here a proof, which seems to be new, and is based on techniques of real JB$^*$-triples.

\begin{lemma}\label{l PeSta generalized} Let $(\mathcal{H}, (.|.))$ be a real Hilbert space, $K$ a compact Hausdorff space, and $\varphi$ a non-zero functional in $C(K,\mathcal{H})^*$. Suppose there exist $t_0\in K,$ $x_0\in S(\mathcal{H}),$ and an open neighborhood $\mathcal{O}$ of $t_0$ satisfying $\varphi (b) = \|\varphi\|$ for every $b\in A(t_0,x_0)$ whose cozero-set is contained in $\mathcal{O}$. Then $\varphi (a) = \|\varphi\| ( a(t_0) | x_0) = \|\varphi\| (x_0^*\otimes \delta_{t_0}) (a) $, for all $a\in C(K,\mathcal{H})$.
\end{lemma}

\begin{proof} Let us assume that $\|\varphi\|=1$. Let $1$ denote the unit element in $C(K)$. Since the element $e= x_0\otimes 1$ is a non-zero tripotent in the real JB$^*$-triple $C(K,\mathcal{H})$ with $e\in A(t_0,x_0)$, it follows from the hypothesis that $\varphi (e) = \|\varphi\| =1$. An application of \cite[Lemma 2.7]{PeSta2001} shows that $\varphi (a) = \varphi P^1 (e) (a)$, for every $a\in C(K,\mathcal{H})$. It is not hard to see that $$\{e,a,e\} (t) = ( e(t) | a(t))  e (t) = ( a(t) | x_0 )  x_0 = x_0 \otimes ( a | x_0),$$ and hence $P^{1} (e) (a) = ( a | x_0 )  x_0 = x_0 \otimes ( a | x_0),$ and $$ \varphi (a) = \varphi P^1 (e) (a) = \varphi (x_0 \otimes ( a | x_0) ) $$ for every $a\in C(K,\mathcal{H})$. This shows that $\varphi= \varphi|_{C(K,\mathbb{R} x_0)}$ can be identified with a norm-one functional in $C(K,\mathbb{R} x_0)^*\cong C(K,\mathbb{R})^*$. The norm-one functional $\psi = \varphi|_{C(K,\mathbb{R} x_0)}\in C(K,\mathbb{R})^*$ satisfies that $\psi (f) =1$ for every $f\in C(K)$ with $\|f\|= 1= f(t_0)$. It is not hard to see, via Urysohn's lemma, that $\ker (\psi)$ contains all $f\in \mathcal{B}_{C(K,\mathbb{R})}$ vanishing on a open neighborhood of $t_0$ contained in $\mathcal{O}$. Therefore, $\psi$ vanishes on every function $f\in C(K)$ with $f(t_0)=0$, and thus $\psi (g) = g(t_0)$ for all $g\in C(K,\mathbb{R})$, and consequently $\varphi (a) = ( a(t_0) | x_0)$, for all $a\in C(K,\mathcal{H})$.
\end{proof}

Accordingly to the notation in \cite{MoriOza2018}, given a face $F$ contained in the unit sphere of a Banach space $X$ and $\lambda\in [-1,1]$ we set $$F_{\lambda} :=\left\{s\in S(X) : \hbox{dist}(x,F)\leq 1-\lambda,\  \hbox{dist}(x,-F)\leq 1+\lambda \right\} $$
$$=\left\{s\in S(X) : \hbox{dist}(x,F)= 1-\lambda,\  \hbox{dist}(x,-F)= 1+\lambda \right\}.$$

Let $p$ be a projection in the bidual, $A^{**}$, of a C$^*$-algebra $A^{**}$. Following  \cite{AkPed92,EdRu96}, we say that $p$ is \emph{compact} if $p$ is closed relative to $A$ (i.e. $A\cap (1-p) A^{**} (1-p)$ is weak$^*$-dense in $(1-p) A^{**} (1-p)$) and there exists a norm-one element $x\in A^{+}$ such that $p\leq x$ (compare \cite[page 422]{AkPed92}). In our setting, for each closed (i.e. compact) subset $C\subseteq K$, the projection $\chi_{_C}$ is compact in $C(K)^{**}$ and rarely lies in $C(K)$. \smallskip

As in \cite{MoriOza2018}, for $\lambda\in [-1,1]$, we define $$F^{A}(p,\lambda) :=\left\{ x\in S(A) : x p = p x = \lambda p\right\}=S(A)\cap \{\lambda p + y : y \in \mathcal{B}_{(1-p)A^{**}(1-p)}\}.$$ We observe that $F^{A}(p,1)=F^{A} (p) = A\cap (p \oplus \mathcal{B}_{(1-p)A^{**}(1-p)})$ is precisely the norm-closed face of $\mathcal{B}_{A}$ associated with the projection $p$ (compare \cite{AkPed92}). In \cite[Lemma 17]{MoriOza2018} it is established that, under these circumstances we have $F^{A}(p,\lambda)= \left(F^{A} (p)\right)_{\lambda}.$ Our next goal is to obtain a version of this fact in the setting of continuous functions valued in a Hilbert space.

\begin{lemma}\label{l  17 for C(K,H)} Let $\mathcal{C}$ be a closed subset of $K$, and let $x_0$ be a norm-one element in $H$. Let $p = \chi_{_\mathcal{C}}$, and let $F_{x_0\otimes p}$ be the set defined in \eqref{eq subtriples F,B,N}. For each $\lambda \in [0,1]$ set $$F(x_0\otimes p,\lambda) :=\left\{ a\in S(C(K,H)) : a p = \lambda x_0\otimes p\right\}.$$ Then $ F(x_0\otimes p,\lambda) = \left(F_{x_0\otimes p}\right)_{\lambda}.$
\end{lemma}

\begin{proof}$(\supseteq)$ Let $a\in \left(F_{x_0\otimes p}\right)_{\lambda}.$ We fix $t_0\in \mathcal{C}$. For each $\varepsilon >0$ there exist $b\in F_{x_0\otimes p}$ and $c\in -F_{x_0\otimes p}$ such that $\|a(t_0) - x_0\|_{_H} \leq \|a-b\|< 1-\lambda +\varepsilon$ and $\|a(t_0) + x_0\|_{_H} \leq \|a-c\|< 1+\lambda +\varepsilon.$ The arbitrariness of $\varepsilon>0$ implies that $\|a(t_0) - x_0\|_{_H} \leq  1-\lambda $ and $\|a(t_0) + x_0\|_{_H} \leq  1+\lambda,$ which proves that $a(t_0) = \lambda  x_0$.\smallskip

$(\subseteq)$ Let us take $a\in F(x_0\otimes p,\lambda)$. To simplify the notation, let us write $E= C(K,H)$. Since $H$ is a (complex) Hilbert space, we can identify $E^{**}$ with the Banach space $C(\tilde{K},(H,w))$ of all continuous functions from $\tilde{K}$ to $H$ when this latter space is provided with its weak topology, where $\tilde{K}$ is a compact Hausdorff space such that $C(K)^{**} \equiv C(\tilde{K})$ (see \cite[Theorem 2]{CamGreim82}).\smallskip

The set $F_{x_0\otimes p}$ is a proper norm-closed face of $\mathcal{B}_{E}$, it is actually the face associated with the compact tripotent $e=x_0\otimes p\in E^{**}\equiv C(\tilde{K},(H,w))$. It has been recently shown in \cite[Theorem 3.6]{BeCuFerPe2018} that the weak$^*$-closure of $F_{x_0\otimes p}$ in $E^{**}$ is precisely the proper weak$^*$-closed of $\mathcal{B}_{C(K,H)^{**}}$ associated with the compact tripotent $e$, that is, $\overline{F_{x_0\otimes p}}^{w^*} = F_{e}^{E^{**}}= e + \mathcal{B}_{E^{**}_0(e)}.$ Clearly, the element $ e + a (1-p)$ belongs to $F_{e}^{E^{**}}= e + \mathcal{B}_{E^{**}_0(e)}$ and $\| a - (e + a (1-p))\|  = \| \lambda x_0\otimes p - x_0\otimes p  \| = 1-\lambda.$ We deduce that $\hbox{dist}(a,\overline{F_{x_0\otimes p}}^{w^*}) \leq 1-\lambda$. Now, an application of the Hahn-Banach separation theorem gives $\hbox{dist}(a,{F_{x_0\otimes p}}) = \hbox{dist}(a,\overline{F_{x_0\otimes p}}^{w^*}) \leq 1-\lambda$. If in the above argument we replace $ e + a (1-p)$ by $ -e + a (1-p)$, we derive $\hbox{dist}(a,-{F_{x_0\otimes p}})\leq 1+\lambda$.
\end{proof}

The following proposition is a first step to obtain a linear extension of a surjective isometry between the unit spheres of $C(K,H)$ and any Banach space $Y$. We shall show that such isometries are affine on the maximal proper faces of $\mathcal{B}_{C(K,H)}$ using an adaptation of the arguments in \cite[Proposition 20]{MoriOza2018}.

\begin{proposition}\label{p Delta is affine in maximal faces}
Let $\Delta : S(C(K,H))\to S(Y)$ be a surjective isometry, where $Y$ is a real Banach space. Suppose $t_0\in K$ and $x_0\in S(H)$. Then there exist a net $(\mathcal{R}_\lambda)_{\lambda}$ of convex subsets of $A(t_0,x_0)$ and a net $(\theta_\lambda)_\lambda$ of affine contractions from $A(t_0,x_0)$ into $\mathcal{R}_{\lambda}$ such that $\theta_\lambda\to \hbox{Id}$ in the point-norm topology. Moreover, for each $\lambda,$ $\mathcal{R}_\lambda$ satisfies the strong Mankiewicz property and $\Delta(\mathcal{R}_\lambda)$ is convex. Consequently $\Delta|_{A(t_0,x_0)}$ is affine.
\end{proposition}

\begin{proof}
	Fix $x_0\in S(H)$ and $t_0\in K$. Let us write $\varphi_0=x_0^*\otimes \delta_{t_0}\in S(C(K,H)^*)$, and consider the norm-closed inner ideal of $C(K,H)$ $$L=\{ b\in C(K,H) : \|b\|_{\varphi_0}=0 \}=\{ b\in C(K,H) : b(t_0)=0 \},$$ where $\|b\|^2_{\varphi_0}= \varphi_0 \J{b}{b}{x_0\otimes 1} $ for each $b\in C(K,H)$. We can always find, via Urysohn's lemma, two nets $(f_\lambda)_\lambda$, $(e_\lambda)_\lambda$ in $C(K)$ satisfying the following properties: $0\leq e_\lambda\leq f_\lambda\leq 1,$ $e_\lambda f_\lambda=e_\lambda$ for every $\lambda\in \Lambda$, $e_\mu \geq e_\lambda$  and $f_\mu \geq f_\lambda$ for every $\mu \geq \lambda$, and $$\|f_\lambda b - b\| \underset{\lambda}{\to}0, \  \|bf_\lambda  - b\| \underset{\lambda}{\to}0, \ \|e_\lambda b - b\| \underset{\lambda}{\to}0, \hbox{ and }  \|be_\lambda  - b\| \underset{\lambda}{\to}0, \quad \forall b\in L.$$ We shall say that $(f_\lambda)_\lambda$ and $(e_\lambda)_\lambda$ are \emph{module-approximate units} for $L$. We can actually assume that each $f_\lambda $ (and hence each $e_\lambda $) vanishes on an open neighborhood of $t_0$.\smallskip
	
We define now $\theta_\lambda: C(K,H)\to C(K,H)$, $\theta_\lambda(c):= x_0\otimes (1-e_\lambda) + c e_\lambda$. Since for each $a\in A(t_0,x_0)$ the element $a-x_0\otimes 1$ belongs to $L$, we deduce that $$\left\|\theta_\lambda(a)  - a\right\| =\left\| (a-x_0\otimes 1) e_{\lambda} - (a -x_0\otimes 1) \right\| \underset{\lambda}{\to}0.$$ Clearly  $\theta_\lambda$ is an affine mapping for every $\lambda$, and $c\in \mathcal{B}_{C(K,H)}$, $\theta_\lambda(c)$ lies in $A(t_0,x_0)$. Finally it is worth noting that $\theta_\lambda$ is contractive.\smallskip
	
From now on we fix a subindex $\lambda$, and thus we shall write $e$, $f$ and $\theta$ for $e_\lambda, f_\lambda$ and $\theta_{\lambda}$, respectively. Let us consider the open subset $\mathcal{O}\subseteq K$ given by $A=(1-f)^{-1}(\RR \setminus \{0\})$. By construction $t_0\in \mathcal{O}$ and $p:=\chi_{_\mathcal{O}}\in C(K)^{**}$ is the range projection of $(1-f)$ in $C(K)^{**}$.\smallskip

We consider next the norm-closed subtriples $F\subseteq  N \subseteq B \subseteq C(K,H)$ defined in (\ref{eq subtriples F,B,N}), that is, \begin{align*} F&=F_{x_0\otimes p}  :=  \{ a\in S(C(K,H)) : ap= x_0 \otimes p \},   \\
B &=B_p :=\{ a\in C(K,H) : ap=h \otimes p, \text{ for some }  h\in H \}, \\
N &=N_p^{x_0} := \{ a\in C(K,H) : ap= \mu x_0 \otimes p, \text{ for some }  \mu\in \CC \}.
\end{align*}
	
Given $a\in A(t_0,x_0)$, we have $\theta(a)(1-f)=x_0\otimes (1-f)$, which proves that $\theta(a)\in F$. We therefore conclude that $\theta|_{A(t_0,x_0)}: A(t_0,x_0)\to F\subseteq A(t_0,x_0)$.\smallskip
	
As we previously commented in section \ref{sec: 3}, we cannot, in general, assume that $p\in C(K)$, so, we shall distinguish the different cases.\smallskip

\noindent\emph{Case 1: We assume that $p\in C(K)$.} In this case we consider the following norm-closed face of $\mathcal{B}_{C(K,H)}$ $$\mathcal{R}=\mathcal{R}_{\lambda}= (x_0\otimes p)\, + \mathcal{B}_{p^\perp C(K,H)}=F\subseteq A(t_0,x_0),$$ where $\mathcal{B}_{p^\perp C(K,H)}\equiv \mathcal{B}_{C((1-p)K,H)}.$ Proposition \ref{p conditions on C(K,H) to satisfy the Strong Mankiewicz property} implies that $\mathcal{R}$ satisfies the strong Mankiewicz property because the translation $\tau_{-x_0\otimes p}$ is surjective affine isometry.\smallskip

It is not hard to see that $$\theta (a) = x_0\otimes (1-e) + a e = x_0\otimes p + (x_0 \otimes (1-e-p)) + a e \in \mathcal{R},$$ and thus $\theta (A(t_0,x_0))\subseteq \mathcal{R}.$\smallskip

Having in mind that $\mathcal{R}$ is an intersection face in the sense employed in \cite[Lemma 8]{MoriOza2018}, the just quoted result implies that $\Delta(F)$ also is an intersection face, and in particular a non-empty convex set. Since $\mathcal{R}$ satisfies the strong Mankiewicz property we deduce that $\Delta|_{\mathcal{R}}$ is affine.\smallskip

\emph{Case 2: We assume that $p\notin C(K)$.} We claim that, in this case, $\overline{\mathcal{O}}\cap (K\setminus \mathcal{O})\neq \emptyset$. Otherwise, $\overline{\mathcal{O}}\cap (K\setminus \mathcal{O})=\emptyset$, and hence $K=\mathcal{O}\mathring{\cup} (K\setminus \mathcal{O})\subseteq \overline{\mathcal{O}}\mathring{\cup} (K\setminus \mathcal{O}) \subseteq K$, which proves that $\mathcal{O}$ is clopen. Therefore $p = \chi_{_\mathcal{O}}$ is continuous, leading to a contradiction.\smallskip
		
Following the construction in Section \ref{sec: 3}, we shall consider the linear mapping $T:B\to H$ given by $T(a)=a(t_0)$ for each $a\in B$. We have seen in Section \ref{sec: 3} that $T$ is a triple homomorphism. Let us now take $a\in B$ and write $a(t_0)=x_a$. By applying that $\overline{\mathcal{O}}\cap (K\setminus \mathcal{O})\neq \emptyset$ we deduce that $\|a\|=\|a|_{(K\setminus \mathcal{O})}\|\geq \|a|_{\mathcal{O}}\|=\|x_a\|$. It follows that $\|T(a)\|=\|a(t_0)\|=\|x_a\|\leq \|a\|$. The arbitrariness of $a\in B$ proves that $T$ is  continuous and contractive.\smallskip
		
Since $N$ is a JB$^*$-triple of $B$ the restriction $T|_N:N\to H$ also is a triple homomorphism, and thus the linear functional $\varphi\equiv x_0^*\circ T|_N:N\to \CC$ is a continuous triple homomorphism. Proposition \ref{p N strong Mankiewicz property} now assures that the closed unit ball of $N^\varphi_\RR:=\varphi^{-1}(\RR)$ satisfies the strong Mankiewicz property.\smallskip
		
In this case we set $$\mathcal{R}_{\lambda}=\mathcal{R}:= (x_0\otimes (1-f)) + f\mathcal{B}_{N_\RR^\varphi}\subseteq A(t_0,x_0).$$ Clearly $\mathcal{R}$ satisfies the strong Mankiewicz property.\smallskip

Let us take $a\in F$. Since $(1-f) a =x_0\otimes (1-f) $, we deduce that $a=(1-f) a +f a= x_0\otimes (1-f) + f a$, with $a p = x_0\otimes p$, $\|a\|= 1$ and $\varphi (a) = 1$. We have therefore shown that $F\subseteq \mathcal{R}$.

Let us show that $\theta (A(t_0,x_0))\subseteq \mathcal{R}$. Namely, for each $a\in A(t_0,x_0)$ we write $$\theta (a) = (x_0\otimes (1-e)) +a e = (x_0\otimes (1-f)) + x_0\otimes (f-e) + a e $$
$$=  (x_0\otimes (1-e))  + f\left(x_0\otimes (1-e) + a e \right),$$ where $p \left(x_0\otimes (1-e) + a e \right) = x_0 \otimes p$ and $\varphi \left(x_0\otimes (1-e) + a e \right) = 1$, which shows that $\theta (a)\in \mathcal{R}$.\smallskip

We shall next show that $\Delta(\mathcal{R})$ is convex. The rest of the proof is just an adaptation of the proof of \cite[Lemma 20]{MoriOza2018}, the argument is included here for completeness.\smallskip

Let us follow the notation in \cite{MoriOza2018}. Given $\gamma\in [-1,1]$ we define $h_{\gamma} : [0,1]\to [-1,1]$, $h_{\gamma} (t) := t + (1-t) \gamma$. For $i\in \{1,2\}$ and $m\in \mathbb{N}$ we set $$G_m^i :=A(t_0,x_0) \bigcap \left(\!\!\bigcap_{\stackrel{k=1}{ \chi_{_{[\frac{2k-2+i}{2m},\frac{2 k -1 +i}{2m}]}}(1-f)\neq0}}^{2m}\!\!\! \!\!\!\! F\left(x_0\otimes \chi_{_{[\frac{2k-2+i}{2m},\frac{2 k -1 +i}{2m}]}}(1-f), h_{\gamma}(\frac{k}{m})\right) \right),$$
and $H_m^i (\gamma) := A(t_0,x_0) \cap \mathcal{N}_{\frac{1}{m}} (G_m^i(\gamma))$, where $\mathcal{N}_{\delta} (G_m^i(\gamma))$ is the $\delta$-neighborhood around $G_m^i(\gamma)$. \smallskip

Given $\gamma_1,\gamma_2\in [-1,1]$ and $\alpha \in [0,1]$, Lemma \ref{l  17 for C(K,H)} and \cite[Lemma 10]{MoriOza2018} assure that for $\gamma_3 = \alpha \gamma_1 + (1-\alpha) \gamma_2$ we have $$\alpha G_m^i (\gamma_1) +(1-\alpha) G_m^i (\gamma_2) \subseteq G_m^i (\gamma_3), $$ and
$$\alpha H_m^i (\gamma_1) +(1-\alpha) H_m^i (\gamma_2) \subseteq H_m^i (\gamma_3).$$

Following the ideas in the proof of \cite[Proposition 20]{MoriOza2018} we shall next show that
\begin{equation}\label{eq Kgamma} \mathcal{R}(\gamma):= \left\{a\in A(t_0,x_0) : p a  = x_0\otimes h_{\gamma} (1-f) \right\} = \bigcap_{m\in \mathbb{N}} \left( H_m^1 (\gamma) \cap H_m^2(\gamma)\right).
\end{equation}

$(\subseteq)$ Consider the function $g_m:[0,1]\to \mathbb{R}$ given by
$$g_m (t):=\left\{%
	\begin{array}{ll}
	\gamma , & \hbox{if $t=0$} \\
    h_{\gamma}(\frac{k}{m}), & \hbox{if $t\in [\frac{2k-1}{2 m},\frac{2 k}{2m}]$ with $k\in\{1,\ldots,m\}$} \\
	\hbox{affine}, & \hbox{in the rest}. \\
	\end{array}%
	\right.$$ By definition $\|g_m-h_{\gamma}\|_{_{C(K)}}\leq \frac{1}{m}$ and $(g_m-h_{\gamma}) (0) =0$, which assures that $(g_m-h_{\gamma}) (1-f) = p (g_m- h_{\gamma}) (1-f)\in p C(K)$. For each $a\in \mathcal{R}(\gamma)$ we have $$ p \Big(a+ x_0\otimes (g_m-h_{\gamma}) (1-f) \Big)= x_0\otimes h_{\gamma} (1-f) + x_0\otimes (g_m-h_{\gamma}) (1-f) $$
$$ = \sum_{k=1}^m x_0\otimes h_{\gamma} (\frac{k}{m}) \chi_{_{[\frac{2k-1}{2 m},\frac{2 k}{2m}]}} (1-f),$$ therefore $b_m:=a+ x_0\otimes (g_m-h_{\gamma}) (1-f)\in G^1_m(\gamma)$ and $\|a - b_m\| = \| x_0\otimes (g_m-h_{\gamma}) (1-f)\| \leq \| g_m-h_{\gamma} \|\leq \frac{1}{m},$ witnessing that $a\in H_m^1(\gamma)$. We can similarly show that $a\in H_m^1(\gamma)$ for every natural $m$.\smallskip

$(\supseteq)$ Take now $a\in \bigcap_{m\in \mathbb{N}} \left( H_m^1 (\gamma) \cap H_m^2(\gamma)\right)$. For each natural $m$, we can find $b_m^i\in G_m^i (\gamma)$ satisfying $\|a-b_m^i\|\leq \frac{1}{m}$. Let us consider the projection $\displaystyle p_m^{i} = \sum_{k=1}^m \chi_{_{[\frac{2k-2+i}{2 m},\frac{2 k-1+i}{2m}]}} (1-f)\in C(K)^{**},$ where $i\in\{1,2\}$. Since $b_m^i\in G_m^i (\gamma)$, we have $\|x_0\otimes h_{\gamma} (1-f) p_m^i -b_m^i p_m^i\|\leq \frac{1}{m},$ and hence $$\|a p_m^i - x_0\otimes h_{\gamma} (1-f) p_m^i \|\leq \frac{2}{m},\hbox{ and } \|a (p_m^1+p_m^2) - x_0\otimes h_{\gamma} (1-f) (p_m^1+p_m^2) \|\leq \frac{2}{m},$$ for every natural $m$. Having in mind that $p_m^1+p_m^2 = \chi_{_{[\frac{1}{2m},1]}}(1-f)$, we deduce that $a p = x_0\otimes h_{\gamma} (1-f)$, which finishes the proof of \eqref{eq Kgamma}.\smallskip

Finally, since clearly $\mathcal{R} =\bigcup_{\gamma\in[-1,1]} \mathcal{R}(\gamma)$ and by \cite[Lemma 11]{MoriOza2018}
$$\Delta^{-1} (\alpha \Delta(\mathcal{R}(\gamma_1)) +(1-\alpha) \Delta(\mathcal{R}(\gamma_2))) \subseteq \bigcap_{m\in \mathbb{N}} \left( H_m^1 (\gamma_3) \cap H_m^2(\gamma_3)\right),$$ for all $\alpha\in [0,1],$ $\gamma_1,\gamma_2\in [-1,1]$ and $\gamma_3 = \alpha \gamma_1 +(1-\alpha) \gamma_2$, we prove that $\Delta(\mathcal{R})$ is convex. \smallskip

Summarizing, we have proved that each $\mathcal{R}_{\lambda}$ satisfies the strong Mankiewicz property, $\Delta(\mathcal{R}_{\lambda})$ is convex, $\theta_{\lambda} (A(t_0,x_0))\subseteq \mathcal{R}_{\lambda}$ and $\|\theta_{\lambda} (a)-a\| \to 0$ for each $a\in A(t_0,x_0)$. Therefore $\Delta|_{\mathcal{R}_{\lambda}}$ is an affine mapping, and consequently, $\Delta|_{A(t_0,x_0)}$ is affine too.
\end{proof}

We shall need a more elaborated discussion on the conclusions of Proposition \ref{p Delta is affine in maximal faces}.

\begin{proposition}\label{p 45 more detailed} Let $\Delta : S(C(K,H))\to S(Y)$ be a surjective isometry, where $Y$ is a real Banach space. Suppose $t_0\in K$ and $x_0\in S(H)$. Then for each $\psi \in Y^*$ there exist $\phi_0$ in $C(K,H)_{\mathbb{R}}^{*}$ and $\gamma_0\in \mathbb{R}$ satisfying $\|\phi_0\|\leq \|\psi\|$ and  $$\psi \Delta (a) = \phi_0 (a) + \gamma_0,\hbox{ for all } a\in A(t_0,x_0).$$
\end{proposition}

\begin{proof}
Let us fix $t_0\in K$ and $x_0\in S(H)$. Let us fix $\psi\in Y^*$. We can assume, without loos of generality, that $\|\psi\|=1$. By Proposition \ref{p Delta is affine in maximal faces} and its proof there exist a net $(\mathcal{R}_\lambda)_{\lambda}$ of convex subsets of $A(t_0,x_0)$ and a net $(\theta_\lambda)_\lambda$ of affine contractions from $A(t_0,x_0)$ into $\mathcal{R}_{\lambda}$ such that $\theta_\lambda\to \hbox{Id}$ in the point-norm topology. Moreover, for each $\lambda,$ $\mathcal{R}_\lambda$ satisfies the strong Mankiewicz property and $\Delta(\mathcal{R}_\lambda)$ is convex. We further know that one of the following statements hold for each $\mathcal{R}_{\lambda}$:\smallskip

\emph{Case 1:} $\mathcal{R}_{\lambda} = (x_0\otimes p) + \mathcal{B}_{(1-p) C(K,H)}$, where $p$ is a projection in $C(K)$. We consider in this case the surjective isometry $\Delta_{\lambda} : \mathcal{B}_{(1-p) C(K, H)}\to \Delta(\mathcal{R}_{\lambda})$ defined by the following diagram:
$$\xymatrix{
	\mathcal{R}_{\lambda} \ar[r]^{\Delta|_{\mathcal{R}_{\lambda}}}   \ar[d]_{\tau_{-x_0\otimes p}}& \Delta(\mathcal{R}_{\lambda}) \\ \mathcal{B}_{(1-p)C(K, H)} \ar@{-->}[ur]_{\Delta_{\lambda}}
}$$ By Proposition \ref{p conditions on C(K,H) to satisfy the Strong Mankiewicz property} and Theorem \ref{t Mori-Ozawa strong Mankiewicz} there exist $c_\lambda\in \Delta(\mathcal{R}_{\lambda})$ and a linear isometry $T_{\lambda} : (1-p) C(K,H)\to Y$ such that $\Delta ( b) = c_\lambda + T_\lambda (b - x_0\otimes p )$ for all $b\in \mathcal{R}_{\lambda}$. By regarding $T_{\lambda}$ as a linear contraction, $\tilde{T}_{\lambda},$ from $C(K,H)$ to $Y$ defined by $\tilde{T}_{\lambda} (a) := T_{\lambda} (a (1-p))$, we deduce that $$\psi \Delta (b) = \gamma_{\lambda} + \phi_{\lambda} (b), \hbox{ for all $b\in \mathcal{R}_{\lambda}$},$$ where $\gamma_{\lambda} = \psi (c_{\lambda})$ is a real number in $[-1,1]$ and $\phi_{\lambda}= \psi\circ \tilde{T}_{\lambda}\in \mathcal{B}_{C(K,H)_{\mathbb{R}}^{*}}$.\smallskip

\emph{Case 2:} $\mathcal{R}_{\lambda} = (x_0\otimes (1-f_{\lambda})) + f_{\lambda} \mathcal{B}_{N^{\varphi}_{\mathbb{R}}}$, where $f_{\lambda}\in S(C(K))$ with $0\leq f_{\lambda}\leq 1$, $p_{\lambda}\in C(K)^{**}\backslash C(K)$ is the range projection of $1-f_{\lambda}$ and $N_{\lambda}$ is the JB$^*$-subtriple of $C(K,H)$ defined by $N_{\lambda} = \{ a\in C(K,H) : a p_{\lambda}= \mu x_0 \otimes p_{\lambda}, \text{ for some }  \mu\in \CC \}$. Furthermore, suppose $p_{\lambda}$ is the characteristic function of the open set $\mathcal{O}_{\lambda}$, then we have $\overline{\mathcal{O}_{\lambda}}\cap (K\setminus \mathcal{O}_{\lambda})\neq \emptyset$. We consider in this case the surjective isometry $\Delta_{\lambda} : f_{\lambda} \mathcal{B}_{N^{\varphi}_{\mathbb{R}}}\to \Delta(\mathcal{R}_{\lambda})$ defined by the following diagram:
$$\xymatrix{
	\mathcal{R}_{\lambda} \ar[r]^{\Delta|_{\mathcal{R}_{\lambda}}}   \ar[d]_{\tau_{-x_0\otimes (1-f_{\lambda})}}& \Delta(\mathcal{R}_{\lambda}) \\ f_{\lambda} \mathcal{B}_{N^{\varphi}_{\mathbb{R}}} \ar@{-->}[ur]_{\Delta_{\lambda}}
}$$ Proposition \ref{p N strong Mankiewicz property} assures that $f_{\lambda} \mathcal{B}_{N^{\varphi}_{\mathbb{R}}}$ satisfies the strong Mankiewicz property, and hence $\Delta_{\lambda}$ is affine. Since $f_{\lambda} (1-p_{\lambda}) =  1-p_{\lambda}$, we can easily deduce that the mapping $\Phi_{\lambda} : a\mapsto f_{\lambda} a$ is a surjective affine isometry from $ \mathcal{B}_{N^{\varphi}_{\mathbb{R}}}$ onto $f_{\lambda} \mathcal{B}_{N^{\varphi}_{\mathbb{R}}}$. Let $z_\lambda:= \Delta(x_0\otimes (1-f_{\lambda})).$ We complete now the previous diagram $$\xymatrix{
	\mathcal{R}_{\lambda} \ar[r]^{\Delta|_{\mathcal{R}_{\lambda}}}   \ar[d]_{\tau_{-x_0\otimes (1-f_{\lambda})}}& \Delta(\mathcal{R}_{\lambda}) \ar[d]^{\tau_{-z_{\lambda}}} \\ f_{\lambda} \mathcal{B}_{N^{\varphi}_{\mathbb{R}}}\ar@{-->}[ur]_{\Delta_{\lambda}} \ar[r]^{F_{\lambda}} \ar[d]_{\Phi^{-1}} & \Delta(\mathcal{R}_{\lambda}) -z_{\lambda} \\ \mathcal{B}_{N^{\varphi}_{\mathbb{R}}}\ar@{-->}[ur]_{T_{\lambda}}
}
$$ The mappings $T_{\lambda}= \Delta_{\lambda} -z_{\lambda}$ and $T_{\lambda}$ are affine and map zero to zero. Let $\tilde{T}_{\lambda} : N_{\mathbb{R}}^{\varphi}\to Y$ be a bounded linear operator whose restriction to $\mathcal{B}_{N^{\varphi}_{\mathbb{R}}}$ is $T_{\lambda}$. Clearly, $\|\tilde{T}_{\lambda}\|\leq 1$. Let $\phi_{\lambda}\in \mathcal{B}_{C(K,H)_{\mathbb{R}}^{*}}$ be a Hahn-Banach extension of $\psi\circ\tilde{T}_{\lambda}\circ \Phi^{-1}\in \left(N_{\mathbb{R}}^{\varphi}\right)^*$. It follows from the previous diagram that $$\psi \Delta(b) = \phi_{\lambda} (b) +\gamma_{\lambda},$$ for every $b\in\mathcal{R}_{\lambda}$, where $\gamma_{\lambda} =- \phi_{\lambda} (x_0\otimes (1-f_{\lambda})) + \psi (z_{\lambda})$ is a real number in the interval $[-2,2]$.\smallskip

We have therefore shown that for each index $\lambda$ there exist a functional $\phi_{\lambda}$ in $\mathcal{B}_{C(K,H)_{\mathbb{R}}^{*}}$ and a real $\gamma_{\lambda}\in[-2,2]$ satisfying \begin{equation}\label{eq  extension 2502} \psi \Delta(b) = \phi_{\lambda} (b) +\gamma_{\lambda}, \hbox{ for every $b\in\mathcal{R}_{\lambda}$. }
\end{equation} Having in mind that $\mathcal{B}_{C(K,H)_{\mathbb{R}}^{*}}$ is weak$^*$-compact (and the compactness of $\mathcal{B}_{\mathbb{R}}$), we can find $\phi_0\in \mathcal{B}_{C(K,H)_{\mathbb{R}}^{*}},$ $\gamma_0\in \mathbb{R}$, and common subnets  $(\phi_{\mu})_{\mu}$ and $(\gamma_{\mu})_{\mu}$ converging to $\phi_0$ and to $\gamma_0$ in the weak$^*$ and norm topologies of $C(K,H)_{\mathbb{R}}^{*}$ and $\mathbb{R}$, respectively. Since, for each $a\in A(t_0,x_0)$ the net $(\theta_\mu (a))_\mu\subseteq \mathcal{R}_{\lambda}$ converges in norm to $a$, we can easily deduce from \eqref{eq  extension 2502} that $\psi \Delta(a) = \phi_{0} (a) +\gamma_{0},$ for every $a\in A(t_0,x_0)$.
\end{proof}

We can now state the main result of this section.

\begin{theorem}\label{t MazurUlam property for C(K,h)}
Let $K$ be a compact Hausdorff space and let $H$ be a complex Hilbert space. Then the Banach space $C(K,H)$ satisfies the Mazur--Ulam property {\rm(}as a real Banach space{\rm)}, that is, for each surjective isometry $\Delta : S(C(K,H))\to S(Y)$, where $Y$ is a real Banach space, there exists a surjective real linear isometry from $C(K,H)$ onto $Y$ whose restriction to $S(C(K,H))$ is $\Delta$.
\end{theorem}

\begin{proof} Let us fix $t_0\in K,$ $x_0\in S(H)$, and $\psi\in \hbox{supp}_\Delta(t_0,x_0)$ (cf. Lemma \ref{l existence of support functionals for the image of a face}). We first observe that if $K=\{t_0\}$, then $C(K,H)$ is isometrically isomorphic to $H$, and thus the desired conclusion follows, for example, from \cite[Proposition 4.15]{BeCuFerPe2018}.\smallskip

We claim that \begin{equation}\label{eq coincidence on extreme points} \psi \Delta (u) = \Re\hbox{e} \langle u(t_0) | x_0\rangle, \hbox{ for all } u \in \partial_e(\mathcal{B}_{C(K,H)}).
\end{equation}

Let us take $t_1\in K\backslash \{t_0\}$ and open neighborhoods $\mathcal{O}_1$, $\mathcal{O}_2$ and $\mathcal{O}_3$ such that $\overline{\mathcal{O}_1}\subset \mathcal{O}_2$, $t_0\in \mathcal{O}_1$, $t_1\in \mathcal{O}_3$, and $\mathcal{O}_2\cap \mathcal{O}_3$. Let $f,g\in C(K)$ whose cozero-sets are contained in $\mathcal{O}_2$ and $\mathcal{O}_3$, respectively, $f(t_0)=1$ and $g(t_1)=1$.  Given $u\in \partial_e(\mathcal{B}_{C(K,H)}),$ Proposition \ref{p 45 more detailed}, applied to the face $A(t_1, u(t_1))$ and $\psi$, implies the existence of a functional $\phi\in \mathcal{B}_{C(K,H)^*_{\mathbb{R}}}$ and a real $\gamma$ satisfying \begin{equation}\label{eq aplication p 46 funcitonals} \psi\Delta (a) = \phi (a) +\gamma, \hbox{ for all } a\in A(t_1,u(t_1)).
\end{equation}

For each $b\in A(t_0,x_0)$, the elements $g u \pm f b\in A(t_0,x_0)$ belong to $A(t_1,u(t_1))$ and to $\pm A(t_0,x_0)$. Combining \eqref{eq aplication p 46 funcitonals}, Lemma \ref{l existence of support functionals for the image of a face} and \eqref{eq post lemma support}, we get $$\pm 1 = \psi \Delta (g u \pm f b) = \pm \phi ( f b) + \phi (g u )+ \gamma = \pm \phi ( f b) + \psi\Delta (g u ).$$ We therefore deduce that $\psi\Delta (g u ) = 0$ and $\phi ( f b)=1$ for every $b\in A(t_0,x_0)$ and every $f$ as above. In particular, $\phi ( f b)=1$ for every $b\in A(t_0,x_0)$ whose cozero-set is contained in $\mathcal{O}_1$. Lemma \ref{l PeSta generalized} assures that $\phi (a) =  \Re\hbox{e} \langle a(t_0) | x_0\rangle $, for all $a\in C(K,\mathcal{H})$. Since $u\in A(t_1,u(t_1))$, \eqref{eq aplication p 46 funcitonals} implies that $\phi (u) =  \Re\hbox{e} \langle u(t_0) | x_0\rangle, $ which finishes the proof of the claim in \eqref{eq coincidence on extreme points}.\smallskip

Now, Corollary \ref{c convex combinations of unitaries in maximal faces} combined with \eqref{eq coincidence on extreme points} and the final conclusion in Proposition \ref{p Delta is affine in maximal faces} prove that $\psi \Delta (a) = \Re\hbox{e} \langle a(t_0)| x_0\rangle,$ for all $a$ in a maximal face of the form $A(t_2,x_2)$ with $t_2\in K, x_2\in S(H).$ Since every $a\in S(C(K,H))$ belongs to a maximal face of the form $A(t_2,x_2)$ with $t_2\in K, x_2\in S(H),$ we conclude that \begin{equation}\label{eq coincidence on the whole sphere 2502} \psi \Delta (a) = \Re\hbox{e} \langle a(t_0)| x_0\rangle, \hbox{ for all } a\in S(C(K,H)).
\end{equation}

Finally, we consider the families $\{ \Re\hbox{e} x_0^*\otimes \delta_{t_0} : t_0\in K, x_0\in S(H)\}\subseteq S(C(K,H)_{\mathbb{R}}^*)$ and $\{ \psi : \psi\in \hbox{supp}_{\Delta}(t_0,x_0), t_0\in K, x_0\in S(H)\}\subseteq S(Y^*)$. Since the first family is norming for $C(K,H)$, the desired conclusion  follows from \eqref{eq coincidence on the whole sphere 2502} and \cite[Lemma 6]{MoriOza2018} (alternatively, \cite[Lemma 2.1]{FangWang06}).
\end{proof}

The conclusion of Theorem \ref{t MazurUlam property for C(K,h)} in the case $H= \mathbb{C}$ is a consequence of \cite[Theorem 1]{MoriOza2018}. The case in which $H$ is a real Hilbert spaces is not fully covered by our theorem. R. Liu proved in \cite[Corollary 6]{Liu2007} that $C(K,\mathbb{R})$ satisfies the Mazur--Ulam property whenever $K$ is a compact Hausdorff space. Let $\mathcal{H} = \ell_2(\Gamma,\mathbb{R})$ be a real Hilbert space with inner product $(\cdot|\cdot)$. Suppose dim$(\mathcal{H})$ is even or infinite. We can write $\Gamma$ as the disjoint union of two subsets $\Gamma_1,\Gamma_2$ for which there exists a bijection $\sigma: \Gamma_1\to \Gamma_2$. Let $H=\ell_2(\Gamma_1)$ denote the usual complex Hilbert space with inner product $\langle\cdot | \cdot\rangle$, and $(H_{\mathbb{R}},\Re\hbox{e} \langle \cdot | \cdot \rangle )$ the underlying real Hilbert space. The mapping $\displaystyle \left(\lambda_j \right)_{j\in\Gamma_1}+ \left(\lambda_{\sigma(j)}\right)_{j\in\Gamma_1} \mapsto \left(\lambda_j +i \lambda_{\sigma(j)}\right)_{j\in\Gamma_1}$ is a surjective real linear isometry from $\mathcal{H}$ onto $H_{\mathbb{R}}$. The next result is a straightforward consequence of our previous Theorem \ref{t MazurUlam property for C(K,h)}.\smallskip

\begin{corollary}\label{c Mazur-Ulam for real Hilber even or infinite dimensional}
Let $K$ be a compact Hausdorff space and let $\mathcal{H}$ be a real Hilbert space with dim$(\mathcal{H})$ even or infinite. Then the real Banach space $C(K,\mathcal{H})$ satisfies the Mazur--Ulam property.
\end{corollary}

There are certain obstacles that prevent to apply the tools developed in Proposition \ref{p N strong Mankiewicz property}, and Lemma \ref{l  17 for C(K,H)} in the case of $C(K,\mathcal{H})$ when $\mathcal{H}$ is a finite-dimensional real Hilbert space with odd dimension. The difficulties in Proposition \ref{p N strong Mankiewicz property} can be solved with Proposition \ref{p N strong Mankiewicz property rel Hilbert spaces}. If in the proof of Lemma \ref{l  17 for C(K,H)}, Theorem \ref{t characterization relatively open weak* closed faces real case} replaces \cite[Theorem 3.6]{BeCuFerPe2018} then the same conclusion holds for real Hilbert spaces. It is a bit more laborious, but no more than a routine exercise, to check that the arguments in Propositions \ref{p Delta is affine in maximal faces} and \ref{p 45 more detailed} and Theorem \ref{t MazurUlam property for C(K,h)} are literally valid to get the following result.

\begin{corollary}\label{c Mazur-Ulam for real Hilber odd finite dimensional}
Let $K$ be a compact Hausdorff space and let $\mathcal{H}$ be a finite-dimen-sional real Hilbert space with odd dimension. Then the real Banach space $C(K,\mathcal{H})$ satisfies the Mazur--Ulam property.
\end{corollary}

The pioneer achievements of  M. Jerison provide generalized versions of the Banach--Stone theorem for spaces of vector-valued continuous functions. Combining Theorem \ref{t MazurUlam property for C(K,h)} and Corollary \ref{c Mazur-Ulam for real Hilber even or infinite dimensional} with the Banach--Stone theorem in \cite[Theorem 7.2.16]{FleJam08} (see also \cite[Definition 7.1.2]{FleJam08}) we obtain next a description of the surjective isometries between the unit spheres of two $C(K,H)$ type spaces.

\begin{corollary}\label{c Jersion sphere} Let $K_1,K_2$ be two compact Hausdorff spaces, let $H$ be a real or complex Hilbert space, and let $Y$ be a strictly convex real Banach space. Suppose $\Delta : S(C(K_1,H))\to S(C(K_2,Y))$ is a surjective isometry. Then there exist a homeomorphism $h:K_2\to K_1$ and a mapping which maps each $t\in K_2$ to a surjective linear isometry $V(t):H\to Y$,
which is continuous from $K_2$ into the space $B(H, Y )$ of bounded linear operators from $H$ to $Y$ with the strong operator topology, such that $$\Delta(a)(t)= V(t) (a(h(t))), \ $$ for all $a\in S(C(K_1,H)), t\in K_2.$ $\hfill \Box$
\end{corollary}

\medskip\medskip

\textbf{Acknowledgements} Authors partially supported by Junta de Andaluc\'{\i}a grant FQM375. First author also supported by Universidad de Granada, Junta de Andaluc{\'i}a and Fondo Social Europeo de la Uni{\'o}n Europea (Iniciativa de Empleo Juvenil), grant number 6087.

\end{document}